\documentclass[12pt]{amsart}
\usepackage[a4paper, left=3cm, right=3cm, height=22cm]{geometry}
\usepackage{amsmath,amsfonts,amsthm, amssymb}
\usepackage[arrow,matrix, all,cmtip]{xy}
\usepackage{enumerate}
\usepackage{cite}
\usepackage[usenames]{color}
\usepackage{amssymb}

\newtheorem{theorem}{Theorem}[section]
\numberwithin{equation}{section}
\newtheorem{prop}[theorem]{Proposition}

\newtheorem{cor}[theorem]{Corollary}

\theoremstyle{remark}

\newtheorem{rmk}[theorem]{Remark}

\newcommand{\mb}{\mathbb}

\newcommand{\ml}{\mathcal}
\newcommand{\mbf}{\mathbf}
\newcommand{\mr}{\mathrm}

\newcommand{\Sf}{\mathbf{S}}

\newcommand{\A}{B} 
\newcommand{\D}{\delta} 
\newcommand{\C}{C} 

\usepackage[pdfauthor={Jin Cao}, %
				pdftitle={Cdga}%
			dvips,colorlinks=true]{hyperref}
\hypersetup{
	bookmarksnumbered=true,
	linkcolor=black,
}
\newcounter{elno}

\def\P{\mathbb{P}}
\def\Z{\mathbb{Z}}                   
\def\Q{\mathbb{Q}}                   
\def\C{\mathbb{C}}                   
\def\N{\mathbb{N}}                   
\def\uhp{{\mathbb H}}                
\def\dR{{\rm dR}}                    
\def\Sp{{\rm Sp}}                    
\def\Gm{\mathbb{G}_m}                 
\def\tr{{{\mathsf t}{\mathsf r}}}                 
\def\spec{{\rm Spec}}            

\newcommand{\mat}[4]{
     \begin{pmatrix}
            #1 & #2 \\
            #3 & #4
       \end{pmatrix}
    }

\def\tmap{{\sf t}}

\def\diag{{\rm diag}}

\newcommand{\Ra}{{\sf R}}
\newcommand{\Tf}{{\sf T}}
\newcommand{\BG}{{\sf G}}

\begin{document}
\author{Jin Cao, Hossein Movasati and Shing-Tung Yau}
\title{Gauss-Manin connection in disguise: Genus two curves}

\address{
Jin Cao\\
Yau Mathematical Sciences Center,
Tsinghua University, Beijing, China
}
\email{caojin@mail.tsinghua.edu.cn}

\address{Hossein Movasati\\
Instituto de Matem\'atica Pura e Aplicada, IMPA, Estrada Dona Castorina, 110, 22460-320, Rio de Janeiro, RJ, Brazil.
}
\email{hossein@impa.br}

\address{Shing-Tung Yau \\
Mathematics Department, Harvard University, 1 Oxford Street, Cambridge MA, 02138 USA.
}
\email{yau@math.harvard.edu}


\begin{abstract}
We describe an algebra of meromorphic  functions on the Siegel domain 
of genus two which contains Siegel modular forms for an arithmetic index six  
subgroup  of the symplectic group  and it is closed under three 
canonical derivations  of the Siegel domain. The main ingredients of our study 
are the moduli of enhanced genus two curves, Gauss-Manin connection and the modular vector
fields living on such moduli spaces.
\end{abstract}
\maketitle
\pagestyle{myheadings}


\section{Introduction}
Igusa in \cite{igusa67}  computes $5$ explicit  generators $E_4,E_6,\chi_{10},\chi_{12},\chi_{35}$ 
of the ring  of Siegel modular forms of genus $2$ for $\Sp(4,\Z)$. The first $4$ are given by 
Eisenstein series and they generate Siegel modular forms of even weight. 
The first example of a differential equation for a Siegel modular form is due to Resnikoff in \cite{Resnikoff1970-1, Resnikoff1970-2}.
He computes in 
\cite[page 496]{Resnikoff1970-1} an eight order differential equation for $E_4$ (in his notation $\psi_4$)  and expresses the  difficulty to find differential equations for $E_6$. 
Bertrand and Zudulin in  \cite{BertrandZudilin2003} show that 
the transcendental degree of the field generated by Siegel modular forms of genus 
$g$  and their derivations is $2g^2+g$. They further describe many differential equations involving theta constants, see also \cite[Theorem 1]{Zudilin2000}. 

We are motivated by the fact that in all these works there is a moduli space equipped with canonical vector
fields which is responsible for all the involved computations. In \cite{ho06-2} it is remarked that the 
Ramanujan's  differential equation between Eisenstein series can be interpreted as a vector field on
the moduli of elliptic curves enhanced with a suitable frame of the first cohomology bundle, see also 
\cite{ho14} for further details.
Such a moduli space in the case of Calabi-Yau varieties is worked out in \cite{ho22, HosseinMurad, ho14} 
and it is 
the  building block of the new theory of  Calabi-Yau modular forms. This is mainly inspired by
many computations in string theory and in particular Yamaguchi and Yau's  polynomials description in 
 \cite{yam04} of generating
functions of genus $g$ Gromov-Witten invariants, see also \cite{alim11} and the references therein.
The case of principally polarized abelian varieties was initiated in \cite{ho18} and the construction
of Ramanujan-type vector fields is done in \cite{ho2020} for a complex moduli framework  and in 
\cite{2019Fonseca} for a moduli stack framework. 
In this paper we work out such a moduli space for 
genus two curves with a marked Weierstrass point and the outcome of our approach 
in terms of Siegel modular forms is the following.

Let $\Gamma$ be the subgroup  of $\Sp(4,\Z)$ generated by four matrices:
\begin{equation}
\label{14.02.2109}
\left(\begin{matrix}
1 & 0 & 0 & 0 \\
0 & 1 & 0 & 0 \\
1 & 0 & 1 & 0 \\
0 & 0 & 0 & 1 
\end{matrix}\right),
\left(\begin{matrix}
1 & 0 & -1 & 1 \\
0 & 1 & 1 & -1 \\
0 & 0 & 1 & 0 \\
0 & 0 & 0 & 1 
\end{matrix}\right), 
\left(\begin{matrix}
1 & 0 & 0 & 0 \\
0 & 1 & 0 & 0 \\
0 & 0 & 1 & 0 \\
0 &1 & 0 & 1 
\end{matrix}\right),
\left(\begin{matrix}
1 & 0 & 0 & 0 \\
0 & 1 & 0 & -1 \\
0 & 0 & 1 & 0 \\
0 & 0 & 0 & 1 
\end{matrix}\right),
\end{equation}
which turns out to be of index $6$, see \S\ref{mono}, 
and $\uhp_2$ be the Siegel upper half plane of genus $2$. We denote an element of $\uhp_2$ by 
$\tau=\begin{pmatrix}\tau_1 &\tau_3\\ \tau_3 &\tau_2\end{pmatrix}$.  
\begin{theorem}
\label{maintheo}
There are meromorphic functions $X_i, \ i=1,2,\ldots, 153$ on $\uhp_2$ with possible poles only along  $\tau_3=0$ such that 
\begin{enumerate}
\item
\label{mt1}
$5$ of $X_i$'s are meromorphic Siegel modular forms  of weights $1,2,3,4,5$ for $\Gamma$.  These are denoted by $X_i=T_{4i},\ i=1,2,3,4,5$,  and we have a quadratic relation of the form $T_{16}^2=T_{20}T_{12}$.
The Siegel modular form $T_{20}$ is holomorphic and  
vanishes only in $\tau_3=0$.
\item
\label{mt2}
The ideal $I$ of all polynomial relations between $X_i$'s is defined over $\Q$ (by the first item we have 
$X_{4}^2-X_{5}X_{3}\in I$).
\item
\label{mt3}
The affine variety $\spec(\Q[X]/I)$ is isomorphic to an open subset of 
the weighted projective space $\P^{6,8,10,3,3,3,3,1,1,1,1}$. Its complement is the  zero set of 
a degree $4$ homogeneous polynomial. 
\item
\label{mt4}
The derivation $\frac{\partial X_i}{\partial \tau_k},\ \ i=1,2,\ldots, k=1,2,3$ multiplied with the 
Siegel modular form $X_5$ are polynomials  
in $X$ with $\Q$ coefficients.
\end{enumerate}
\end{theorem}
The proof of Theorem \ref{maintheo}, and also the structure of the paper is as follows.
In \S\ref{s1} we first use some classical statements on the moduli of genus two curves in order
to construct the moduli of genus two curves endowed with a Weierstrass point. Then using a well-known basis of de Rham cohomologies, we compute
the Gauss-Manin connection over such a moduli. In \S\ref{CP} we use the hypercohomology definition of de Rham cohomology 
and compute the cup product in de Rham cohomology. The content of these two sections are needed
in order to construct the moduli space $\Tf$ of genus two curves enhanced with a basis of de 
Rham cohomologies with some compatibility conditions. 
This is done in \S\ref{s3}. In \S\ref{s4} using the Gauss-Manin connection matrix
on $\Tf$, we construct three vector fields
$\Ra_k,\ \ k=1,2,3$ in $\Tf$ which will be eventually interpreted as derivations $\frac{\partial}{\partial \tau_k},\  k=1,2,3$ in \S\ref{s5}. The bridge between regular functions in $\Tf$, and  
meromorphic functions $X_i$ in the Siegel domain is the $\tmap$-map which is explained in \S\ref{s5}. 
Our functions 
$X_i$ have functional equations with respect to the underlying monodromy group  and  in 
\S\ref{s6} we compute such a group explicitly. This turns out to be the group $\Gamma$. We would like to highlight that instead of genus two curves we could start with the Clingher-Doran  family of K3 surfaces in \cite{do10-1}, as some of the ingredients of our work has been worked out in \cite{GMCD-K3}. The computations in this case become heavier and this will be exploited in a future work.  

The main body of the  paper was written during the visit of the first two authors in
CMSA, Harvard university. They would like to thank the institute for financial support and 
its lovely and stimulating ambient. The last section was worked out during the first two authors's stays at the Issac Newton Institute for Mathematical Sciences, Cambridge university. They are very grateful to the institution for the hospitality and financial support.

\section{A geometric set up}
\label{s1}
\subsection{Hyperelliptic curves}
Fix a positive integer $d$. 
Given $t = (t_1, \cdots, t_d) \in V:=\mb{C}^d$, we let $L_t$ be the affine curve defined by 
$y^2 = f(x)$, where 
\begin{equation} \label{P}
f(x) = x^d+\sum_{k=1}^d 
t_{k} x^{d - k}.
\end{equation}
We set: $\deg x = 2, \deg y = d$ and in this way $P$ is a tame polynomial in the sense of \cite[Chapter 10]{ho13}.
Let $\alpha_i,\ \ i = 1 \cdots d$ be the formal roots of $f(x)$, i.e., 
$f(x) = \prod_{i = 1}^d(x -\alpha_i)$, and then we define the discriminant $\Delta$ of $f(x)$ as
\begin{equation}
\Delta = \prod_{1 \leq i < j \leq d} (\alpha_i - \alpha_j)^2.
\end{equation}
Because of the fundamental theorem of symmetric polynomials, it is known that 
$\Delta \in \Q[t]$.
Hence, if $t \in V-\{ \Delta = 0\}$ then $L_t$ is a smooth curve.

From now on assume that $d= 2g+1$ is odd. 
For $t \in V$, we let $Y_t$ be the completion of $L_t$ which is obtained by adding a point at infinity 
$\infty:=[0:1:0]\in\P^2$ to $L_t$.  The projective curve $Y_t$ is smooth of genus $g$ 
which is a hyperelliptic curve. 
It is known that  $H^1_{dR}(L_t)\cong H^1_{dR}(Y_t)$ has a basis consisting of
\begin{equation}\label{de Rham basis1}
\left[\frac{\mr{d}x}{y}\right], \left[\frac{x\mr{d}x}{y}\right], \cdots, \left[\frac{x^{2g-1}\mr{d}x}{y}\right],
\end{equation}
from which the first $g$ elements form a basis of $H^0(Y_t, \Omega^1_{Y_t})$.

Assume that $d$ is even, say $2g+2$.  The completion of (\ref{P}) also gives us a hyperelliptic curve $Y_t$. However, $L_t$ is corresponding to an open subvariety of $Y_t$ by removing two points. In this case,  $H^1_{dR}(L_t)$ also has a basis consisting of
\begin{equation}
\left[\frac{\mr{d}x}{y}\right], \left[\frac{x\mr{d}x}{y}\right], \cdots, \left[\frac{x^{2g}\mr{d}x}{y}\right].
\end{equation}
The first $g$ elements are holomorphic, but the rest may have non-vanishing residues. In order to get a basis of $H^1_{dR}(Y_t)$, we need to find $2g$ elements in $H^1_{dR}(L_t)$ without residues. This basis can be found after choosing a coordinate function around the two points at infinity,  \S\ref{3oct2019}. As an example for
 \[
 y^2 = f(x) = x^6 + t_2x^4 + t_3x^3 + t_4x^2 + t_5x + t_6.
 \]
we have
 \[
 \mr{\rm Res}_{\infty} \frac{x^2\mr{d}x}{y} = -1, \ \ \mr{\rm Res}_{\infty}\frac{x^3\mr{d}x}{y} = 0,\ \  \mr{\rm Res}_{\infty} \frac{x^4\mr{d}x}{y} = \frac{t_2}{2},
 \]
 where $\infty$ is one of the two points at infinity, 
 and so, we may choose the following  basis of $H_{dR}^{1}(Y_t)$:
 \[
 \frac{\mr{d}x}{y}, \ \ \frac{x\mr{d}x}{y},\ \  \frac{x^3 \mr{d}x}{y},\ \  \frac{t_2}{2}  \frac{x^2\mr{d}x}{y} + \frac{x^4 \mr{d}x}{y}.
 \]

\subsection{Gauss-Manin connections}
We consider the family of hyperelliptic curves $Y_t$ over $V$ and let $H^1_{dR}(Y/V)$ be the first 
cohomology bundle, that is, its fiber over $t\in V$ is $H^1_{dR}(Y_t)$.  By abuse of notation, we use $\omega_i$ to denote the global section of $H^1_{dR}(Y/V)$, whose value at $t$ is the $[\frac{x^{i}\mr{d}x}{y}]$ of $H^1_{dR}(Y_t)$. After fixing these sections, the Gauss-Manin connection 
\[
\nabla: H^1_{dR}(Y/V) \to \Omega^1_V \otimes_{\ml{O}_V} H^1_{dR}(Y/V)
\]
can be expressed as: 
\begin{equation}  
\nabla \left(\begin{matrix}
\omega_1 \\
\omega_2 \\
\cdots  \\
\omega_{2g}
\end{matrix}\right) = \left(\begin{matrix}
b_{1,1} & b_{1,2} & \cdots & b_{1, 2g} \\
b_{2,1} & b_{2,2} & \cdots & b_{2, 2g} \\
\cdots \\
b_{2g, 1} &b_{2g,2} &\cdots & b_{2g, 2g} 
\end{matrix}\right) \otimes \left(\begin{matrix}
\omega_1 \\
\omega_2 \\
\cdots  \\
\omega_{2g}
\end{matrix}\right),
\end{equation}
where $b_{i, j}$ are meromorphic differential $1$-forms in $t_1, \cdots, t_d$ with pole order one along $\Delta=0$. 
If we denote the matrix in the right hand side by $\A$, then we may write $\A$ as
\begin{equation}
\A = \A_1 \mr{d}t_1 + \cdots + \A_d \mr{d}t_d,
\end{equation}
where $\A_i$ are $2g \times 2g$ matrices with entries lying in 
$\frac{1}{\Delta} \mb{Q}[t_1, t_2, \cdots, t_d]$.
For simplicity, we decompose the Gauss-Manin connection matrix $B$ as 
\begin{equation}
\left(\begin{matrix}
B_{1} & B_{2} \\
B_{3} & B_{4}
\end{matrix}\right).
\end{equation}
where $B_i$'s are $g\times g$ matrices. We also use the notation 
$B_ i = \sum^d_{m = 1}  B_{m,i}\cdot \mr{d}t_m$, where $B_{m, i}$ are $g \times g$-matrices.
In general the expressions of Gauss-Manin matrices are huge. However, it turns out that 
the inverses of Gauss-Manin matrices are much simpler. Below is the data of the Gauss-Manin connection for 
a family of hyperelliptic  curves of genus $2$ given by 
\begin{equation}
\label{quintic}
Y_t: \ \ y^2= x^5 + t_2x^3+t_3x^2 + t_4x + t_5.
\end{equation}
These are computed by the procedure {\tt gaussmaninmatrix} of the library 
{\tt foliation.lib}
\footnote{\tt http://w3.impa.br/$\sim$hossein/foliation-allversions/foliation.lib }
of {\sc Singular}, \cite{GPS01}, and the relevant algorithms are
explained in \cite[Chapter 4]{ho06-1}, see also \cite[Chapter 12]{ho13}. 

\begin{eqnarray}
\label{23.02.2019} & & \\ \nonumber
5^5 \cdot \Delta &=&108t_{2}^{5}t_{5}^{2}-72t_{2}^{4}t_{3}t_{4}t_{5}+16t_{2}^{4}t_{4}^{3}+16t_{2}^{3}t_{3}^{3}t_{5}-4t_{2}^{3}t_{3}^{2}t_{4}^{2}-900t_{2}^{3}t_{4}t_{5}^{2}+825t_{2}^{2}t_{3}^{2}t_{5}^{2}  \\ \nonumber  & &
+560t_{2}^{2}t_{3}t_{4}^{2}t_{5}-128t_{2}^{2}t_{4}^{4}-630t_{2}t_{3}^{3}t_{4}t_{5}+144t_{2}t_{3}^{2}t_{4}^{3}-3750t_{2}t_{3}t_{5}^{3}+2000t_{2}t_{4}^{2}t_{5}^{2} \\ \nonumber  & & +108t_{3}^{5}t_{5}-27t_{3}^{4}t_{4}^{2} +2250t_{3}^{2}t_{4}t_{5}^{2}-1600t_{3}t_{4}^{3}t_{5}+256t_{4}^{5}+3125t_{5}^{4},
\end{eqnarray}

$$
B^{-1}_2 =\left(
\begin{array}{cc}
-\frac{72t_{2}^{2}t_{5}^{2}-56t_{2}t_{3}t_{4}t_{5}+12t_{2}t_{4}^{3}+16t_{3}^{3}t_{5}-4t_{3}^{2}t_{4}^{2}}{24t_{2}t_{5}^{2}-12t_{3}t_{4}t_{5}+3t_{4}^{3}} & \frac{6t_{2}^{2}t_{4}t_{5}-8t_{2}t_{3}^{2}t_{5}+2t_{2}t_{3}t_{4}^{2}+80t_{3}t_{5}^{2}-30t_{4}^{2}t_{5}}{24t_{2}t_{5}^{2}-12t_{3}t_{4}t_{5}+3t_{4}^{3}}  \\
-\frac{24t_{2}t_{3}t_{5}^{2}-2t_{2}t_{4}^{2}t_{5}-8t_{3}^{2}t_{4}t_{5}+2t_{3}t_{4}^{3}}{24t_{2}t_{5}^{2}-12t_{3}t_{4}t_{5}+3t_{4}^{3}} & -\frac{36t_{2}^{2}t_{5}^{2}-16t_{2}t_{3}t_{4}t_{5}+4t_{2}t_{4}^{3}-20t_{4}t_{5}^{2}}{24t_{2}t_{5}^{2}-12t_{3}t_{4}t_{5}+3t_{4}^{3}}  \\
\frac{152t_{2}t_{4}t_{5}^{2}+80t_{3}^{2}t_{5}^{2}-116t_{3}t_{4}^{2}t_{5}+24t_{4}^{4}}{24t_{2}t_{5}^{2}-12t_{3}t_{4}t_{5}+3t_{4}^{3}} & \frac{184t_{2}t_{3}t_{5}^{2}-10t_{2}t_{4}^{2}t_{5}-72t_{3}^{2}t_{4}t_{5}+18t_{3}t_{4}^{3}-400t_{5}^{3}}{24t_{2}t_{5}^{2}-12t_{3}t_{4}t_{5}+3t_{4}^{3}} \\
\frac{10t_{5}}{3} & \frac{8t_{4}}{3} 
\end{array} \right.
$$
$$
\left.\begin{array}{cc}
 -\frac{40t_{2}t_{5}^{2}-28t_{3}t_{4}t_{5}+8t_{4}^{3}}{8t_{2}t_{5}^{2}-4t_{3}t_{4}t_{5}+t_{4}^{3}} & -\frac{6t_{2}t_{4}t_{5}-8t_{3}^{2}t_{5}+2t_{3}t_{4}^{2}}{24t_{2}t_{5}^{2}-12t_{3}t_{4}t_{5}+3t_{4}^{3}} \\
\frac{2t_{4}^{2}t_{5}}{8t_{2}t_{5}^{2}-4t_{3}t_{4}t_{5}+t_{4}^{3}} & -\frac{60t_{2}t_{5}^{2}-32t_{3}t_{4}t_{5}+8t_{4}^{3}}{24t_{2}t_{5}^{2}-12t_{3}t_{4}t_{5}+3t_{4}^{3}} \\
\frac{32t_{2}^{2}t_{5}^{2}-16t_{2}t_{3}t_{4}t_{5}+4t_{2}t_{4}^{3}-40t_{4}t_{5}^{2}}{8t_{2}t_{5}^{2}-4t_{3}t_{4}t_{5}+t_{4}^{3}} & -\frac{40t_{3}t_{5}^{2}-10t_{4}^{2}t_{5}}{24t_{2}t_{5}^{2}-12t_{3}t_{4}t_{5}+3t_{4}^{3}} \\
2t_{3} & \frac{4t_{2}}{3}
\end{array}\right),
$$

$$
B^{-1}_3 = 
\left(\begin{array}{cc}
\frac{6t_{2}t_{4}t_{5}-8t_{3}^{2}t_{5}+2t_{3}t_{4}^{2}}{4t_{3}t_{5}-t_{4}^{2}} & -\frac{12t_{2}t_{3}t_{5}-4t_{2}t_{4}^{2}}{4t_{3}t_{5}-t_{4}^{2}}\\
\frac{60t_{2}t_{5}^{2}-32t_{3}t_{4}t_{5}+8t_{4}^{3}}{4t_{3}t_{5}-t_{4}^{2}} & \frac{10t_{2}t_{4}t_{5}-24t_{3}^{2}t_{5}+6t_{3}t_{4}^{2}}{4t_{3}t_{5}-t_{4}^{2}} \\
10t_{5} & 8t_{4} \\
-\frac{32t_{2}^{2}t_{5}^{2}-16t_{2}t_{3}t_{4}t_{5}+4t_{2}t_{4}^{3}}{4t_{3}t_{5}-t_{4}^{2}} & -\frac{16t_{2}^{2}t_{4}t_{5}-32t_{2}t_{3}^{2}t_{5}+8t_{2}t_{3}t_{4}^{2}-40t_{3}t_{5}^{2}+10t_{4}^{2}t_{5}}{12t_{3}t_{5}-3t_{4}^{2}} 
\end{array} \right.
$$
$$
\left. \begin{array}{cc}
\frac{10t_{4}t_{5}}{4t_{3}t_{5}-t_{4}^{2}} & -\frac{20t_{3}t_{5}-8t_{4}^{2}}{4t_{3}t_{5}-t_{4}^{2}} \\
-\frac{16t_{2}t_{3}t_{5}-4t_{2}t_{4}^{2}-100t_{5}^{2}}{4t_{3}t_{5}-t_{4}^{2}} & \frac{30t_{4}t_{5}}{4t_{3}t_{5}-t_{4}^{2}} \\
6t_{3} & 4t_{2} \\
\frac{16t_{2}^{2}t_{3}t_{5}-4t_{2}^{2}t_{4}^{2}-160t_{2}t_{5}^{2}+32t_{3}t_{4}t_{5}-8t_{4}^{3}}{12t_{3}t_{5}-3t_{4}^{2}} & -\frac{16t_{2}t_{4}t_{5}-8t_{3}^{2}t_{5}+2t_{3}t_{4}^{2}}{4t_{3}t_{5}-t_{4}^{2}}
\end{array}\right), 
$$
$$
B^{-1}_4 = \left(
\begin{array}{*{4}{c}}
\frac{20t_{3}t_{5}-8t_{4}^{2}}{3t_{4}} & \frac{10t_{2}t_{5}-2t_{3}t_{4}}{t_{4}} & -\frac{4t_{2}}{3} & \frac{50t_{5}}{3t_{4}} \\
-10t_{5} & -8t_{4} & -6t_{3} & -4t_{2} \\
-\frac{16t_{2}t_{3}t_{5}-4t_{2}t_{4}^{2}}{3t_{4}} & -\frac{8t_{2}^{2}t_{5}-10t_{4}t_{5}}{t_{4}} & -\frac{4t_{2}^{2}-24t_{4}}{3} & -\frac{40t_{2}t_{5}-18t_{3}t_{4}}{3t_{4}} \\
\frac{16t_{2}t_{4}t_{5}-8t_{3}^{2}t_{5}+2t_{3}t_{4}^{2}}{3t_{4}} & -\frac{4t_{2}t_{3}t_{5}-4t_{2}t_{4}^{2}}{t_{4}} & \frac{6t_{2}t_{3}+10t_{5}}{3} & \frac{4t_{2}^{2}t_{4}-20t_{3}t_{5}+8t_{4}^{2}}{3t_{4}}
\end{array}
\right),
$$
$$
B^{-1}_5 = \left(
\begin{array}{*{4}{c}}
-\frac{10t_{5}}{3} & -\frac{8t_{4}}{3} & -2t_{3} & -\frac{4t_{2}}{3} \\
\frac{4t_{2}t_{4}}{5} & \frac{8t_{2}t_{3}-50t_{5}}{5} & \frac{12t_{2}^{2}-40t_{4}}{5} & -6t_{3} \\
\frac{40t_{2}t_{5}-18t_{3}t_{4}}{15} & \frac{20t_{2}t_{4}-36t_{3}^{2}}{15} & -\frac{18t_{2}t_{3}-50t_{5}}{5} & -\frac{4t_{2}^{2}-24t_{4}}{3} \\
-\frac{4t_{2}^{2}t_{4}-20t_{3}t_{5}+8t_{4}^{2}}{15} & -\frac{8t_{2}^{2}t_{3}-80t_{2}t_{5}+6t_{3}t_{4}}{15} & -\frac{4t_{2}^{3}-12t_{2}t_{4}}{5} & \frac{6t_{2}t_{3}+10t_{5}}{3}
\end{array}
\right).
$$
We could also compute the Gauss-Manin connection matrices of the family 
$y^2=x^6+t_2x^4+t_3x^3+t_4x^2+t_5x+t_6$ which is too big to fit into this paper.
For the computer data of this see 
\href{http://w3.impa.br/~hossein/WikiHossein/files/Singular%20Codes/2019-05-GaussManinConnectionDegree6HyperElliptic.txt}
{the second author's webpage}.

\section{Cup product in de Rham cohomology} 
\label{CP}
In $H^1_\dR(Y_t)$ we have a natural pairing which is $\langle \alpha,\beta\rangle:=
\frac{1}{2\pi i}\int_{Y_t}\alpha\cup \beta$ for $\alpha,\beta\in H^1_\dR(Y_t)$.  
In this section we want to compute this pairing in the basis $\omega_i$:
\begin{equation}
\label{15feb2019}
\Omega:=[\langle \omega_i,\omega_j\rangle]=
\begin{bmatrix}
 \Omega_1 & \Omega_{2}\\
\Omega_3 & \Omega_4
\end{bmatrix}= 
\begin{bmatrix}
 0 & \Omega_{2}\\
 -\Omega_2^{\tr} & \Omega_4
\end{bmatrix}, 
\end{equation}
where $\Omega_i$'s are $g\times g$ matrices. 
Applying the Guass-Manin connection on $\Omega$, we get
\begin{equation}
\label{miami2019}
d \Omega = B \Omega  + \Omega  B^{\rm{tr}},
\end{equation}
where $B$ is the Gauss-Manin connection matrix written in the basis $\omega_i$. 
In order to get entries of $\Omega$, one needs to solve the above equation. 
This seems to be a quite difficult task. In this section we give a direct way to compute $\Omega$.

\subsection{Coordinate functions at the infinity point}
\label{3oct2019}
We write $Y_t$ in homogeneous coordinates: $y^2 z^{d-2} = F(x, z)$, where $F(x, 1) = f(x)$ and
cover it with two open sets: $U_0:=\{z\not=0\}$ and $U_1:=\{y\not=0\}$. 
It is easy to see that $[0:1:0]$ is the only possible singular point of  $Y_t$. It lies in   
$U_1$. We do blow-ups to determine the coordinate function at the desingularization of this point. We first perform 
$(x, z)=(x_1, t_1x_1)$. Then  $Y_t$ in $U_1$ is given by
\[
(t_1x_1)^{d-2} = F(x_1, t_1x_1), \ \ \hbox{ that is },\ \   t_1^{d-2} = x_1^2 F(1, t_1).
\]
For the next, we do the blow-up: $(x_1, t_1)=(x_2 t_2, t_2)$. 
Then we get:
\[
t_2^{d-2} = (x_2 t_2)^2 F(1,  t_2), \ \  \hbox{ that is },\ \  t_2^{d-4} = x_2^2 F(1, t_2).
\]
We do more blow-ups.  If $d$ is odd, we will stop with an equation 
$t_{\frac{d-1}{2}} = x_{\frac{d-1}{2}}^2F(1, t_{\frac{d-1}{2}})$, whose linear term is $t_{\frac{d-1}{2}}$.  We may choose $x_{\frac{d-1}{2}}$ as the coordinate function. Pulling pack along all these blow-up maps and pulling back along the transformation map from the $U_0$-chart to the $U_1$-chart, we get the coordinate function $t = \frac{x^{\frac{d-1}{2}}}{y}$ written in the coordinate system of $U_0$. 
If $d$ is even, we will stop with an equation $t_{\frac{d-2}{2}}^2 = x_{\frac{d-2}{2}}^2F(1, t_{\frac{d-2}{2}})$ which shows that $Y_t$ has two points at infinity. 
In this case, we may  choose $x_{\frac{d-2}{2}}$ as a coordinate function. Pulling back will lead us to the coordinate function $t = \frac{x^{\frac{d-2}{2}}}{y}$.

\subsection{De Rham cohomology as hypercohomology} \label{de Rham basis}
In this subsection we focus on our genus two curve $Y_t$. 
The algebraic de Rham cohomology of $Y_t$ can be described as:
\begin{equation} \label{H1}
H^1_{dR}(Y) \cong \frac{ \{(\omega_0, \omega_1)\in \Omega^1_{U_0} \times \Omega^1_{U_1} | \omega_1 - \omega_0 \in \mr{d}(\Omega^0_{U_0 \cap U_1}) \} }{\mr{d}\Omega^0_{U_0} \times \mr{d}\Omega^0_{U_1} },
\end{equation}
and
\begin{equation} \label{H2}
H^2_{dR}(Y) \cong \frac{\Omega^1_{U_0 \cap U_1}}{ \Omega^1_{U_0} + \Omega^1_{U_1} + \mr{d} \Omega^0_{U_0 \cap U_1}},
\end{equation}
see for instance \cite[\S 2.6]{ho14}.  
Under the first isomorphisms (\ref{H1}), the basis $\omega_i = \frac{x^i \mr{d}x}{y} (i = 0, \cdots, 3)$ can 
be represented as
\begin{equation}
\omega_i \to \begin{cases}
(\frac{x^i \mr{d}x}{y}|_{U_0}, \frac{x^i \mr{d}x}{y}|_{U_1}), \ 
\ \ \ \ \ \ \ \ \ \ \ \ \ \ \ \ \ \ \  \ i = 0, 1, \\
(\frac{x^i \mr{d}x}{y}|_{U_0}, \frac{x^i \mr{d}x}{y}|_{U_1} + \mr{d}(P_i(\frac{x^2}{y}))), \ \ \ \ 
\ i = 2, 3, \\
\end{cases}
\end{equation}
for some polynomials $P_i$ such that the underlying sum becomes holomorphic at infinity. 
In the next subsection, we explain how to compute $P_2, P_3$.

\subsection{Computing the cup product}
In the  $U_1:=\{y\neq 0\}$, the curve $Y_t$   is given by:
\[
z^3 = x^5 + t_2 x^3z^2 + t_3x^2z^3+ t_4xz^4 + t_5z^5.
\]
In this chart, $\frac{x^i \mr{d}x}{y}$ becomes $\frac{x^i \mr{d}x}{z^i} - 
\frac{x^{i+1} \mr{d}z}{z^{i+1}}$. 
We may choose the coordinate function $t = \frac{x^2}{y}$ on $U_0$ as explained at the beginning of this section. On the $U_0$-chart, along $t = 0$, we may compute the local $t$-expansion of $(x, y)$: assume that
\begin{equation}
x = t^{-2} + a_{-1} t^{-1} + a_0 + a_1 t + a_2 t^2 + \cdots,
\end{equation}
then
\begin{equation}
y = t^{-5} + (\frac{5}{2})t^{-4}+(\frac{15}{8}a^2_{-1} + \frac{5}{2}a_0)t^{-3} + (\frac{5}{16}a^3_{-1} + \frac{15}{4}a_{-1}a_0 + \frac{5}{2}a_1)+ \cdots. 
\end{equation}
Under the coordinate function $t = \frac{x^2}{y}$, we may solve these coefficients 
\[
a_{-1} = a_0 = a_1 = 0, a_2 = -t_2, a_3 = 0, \ldots.
\]
Therefore, for $\frac{x^2 \mr{d}x}{y}$, we may choose a representative $\frac{x^2 \mr{d}x}{y} - 2 \mr{d}((\frac{x^2}{y})^{-1})$, which is algebraic in $U_1$. Similarly for $\frac{x^3 \mr{d}x}{y}$, we may choose a representative $\frac{x^3 \mr{d}x}{y} - \frac{2}{3} \mr{d}((\frac{x^2}{y})^{-3})$. Hence, we have:
$$
P_2(t) = -2 t^{-1},\ \ \ \ P_3(t) = -\frac{2}{3} t^{-3}. 
$$
According to the cup product formula given in \cite[\S 2.10]{ho14}, for any $0 \leq i, j \leq 3$, we have:
\begin{equation}
\frac{x^i \mr{d}x}{y} \cup \frac{x^j \mr{d}x}{y} = P_j(\frac{x^2}{y})\frac{x^i \mr{d}x}{y} - P_i(\frac{x^2}{y})\frac{x^j \mr{d}x}{y} + P_i \mr{d}P_j,
\end{equation}
and 
\begin{equation}
\begin{split}
&\left\langle \frac{x^i \mr{d}x}{y}, \frac{x^j \mr{d}x}{y} \right\rangle = \hbox{ the \ residue \ of\ } P_j(\frac{x^2}{y})\frac{x^i \mr{d}x}{y} - P_i(\frac{x^2}{y})\frac{x^j \mr{d}x}{y} + P_i \mr{d}P_j \ \hbox{ at  } \infty.
 \end{split}
\end{equation}
We have  $P_0 = P_1 = 0$ and $P_2, P_3$ are defined above and so we get:
\begin{equation} \label{Phi-omega}
\Omega =  \left(\begin{matrix}
0  & 0  & 0 & \frac{4}{3} \\
0  & 0  & 4 & 0 \\
0  & -4 & 0 &  \frac{4}{3} t_2 \\
-\frac{4}{3} & 0 & -\frac{4}{3} t_2 & 0
\end{matrix}\right).
\end{equation}
It can be easily verified that $\Omega$ satisfies the differential equation \eqref{miami2019}. 

\section{Moduli spaces} \label{s3}
\subsection{Moduli space I}
We recall that the moduli space $\ml{M}_2$ of curves of genus $2$ can be described as the GIT quotient associated to the action of $GL_2(\mb{C})$ on the space of binary sextics with non-vanishing discriminants, see for instance \cite[\S 4]{2017vanGeer} and the appendix of \cite{1998Dan}. 
Precisely, such a binary sextics $\tilde{Q}(x_1, x_2)$ determines an affine equation 
$C: y^2 = Q(x)$, where $Q(x) = \tilde{Q}(x, 1)$, which is a curve of genus $2$. 
Moreover, the curve $C$ comes 
with a basis of differentials $\frac{\mr{d}x}{y}$ and $\frac{x\mr{d}x}{y}$. 
The group $GL_2(\mb{C})$ acts on these data by
\[
x \to \frac{ax+b}{cx+d}, y \to \frac{y}{(cx+d)^3}, 
\left(\begin{matrix}
\frac{\mr{d}x}{y}  \\
 \frac{x\mr{d}x}{y}
\end{matrix}\right) \to \det(A) \left(\begin{matrix}
d & c \\
b & a
\end{matrix}\right)\left(\begin{matrix}
\frac{\mr{d}x}{y}  \\
 \frac{x\mr{d}x}{y}
\end{matrix}\right).
\]
The six zeros of $Q$ with $y=0$ gives us $6$ Weierstrass points of $C$. Now we assume that $x$ divides $Q(x)$. After a change  of variables:
\[
x\to \frac{1}{x}, \ \ \ y\to \frac{y}{x^3},
\]
we get a curve $y^2=f(x)$, where $f(x)=x^6Q(\frac{1}{x})$ is a polynomial of degree $5$. 
The subgroup $B$ of $GL_2(\mb{C})$ fixing the root $x=0$ of $f(x)$, consists of lower triangular matrices of the form $\left(\begin{matrix}
a &0 \\
c &d
\end{matrix}\right)$. It acts on the parameter space of $y^2 = f(x)$ by
\begin{equation} \label{action}
x \to \frac{c+dx}{a}, y \to \frac{y}{a^3}, 
\end{equation}
which is compatible with its action on the corresponding sextic. 
Notice that each orbit of $y^2 = f(x)$ under the action of 
$B$ can be chosen as \eqref{quintic}. 
Therefore,  giving a quintic with non-vanishing discriminant will lead to a 
$GL_2(\mb{C})$-orbit of the equation $y^2 = x P(x)$  which is just a genus $2$ curve with a Weierstrass point $(0, 0)$. For the family 
\eqref{quintic} the action of $B$ reduces to the action of  
$\mb{G}_m = \mb{C}^*$ given by:
\begin{equation} \label{ta}
(x, y) \to (a^2 x, a^5 y), t \to t \bullet a: (t_2, t_3, t_4, t_5) \to (a^{-4}t_2, a^{-6}t_3, a^{-8}t_4, a^{-10}t_5)
\end{equation}
for $a \in \mb{G}_m$. The curve \eqref{quintic} after  the change of variables $y\to \frac{y}{x^3},\ x\to \frac{1}{x}$ has the form $y^2=t_5x^6+t_4x^5y+t_3x^4y^2+t_2x^3y^3+xy^5$, and so, for the moduli of genus two curve we have to remove $t_5=0$. 
We conclude that
\begin{prop}
\label{23feb2019}
The moduli space $\ml{M}_2(w)$ of hyperelliptic curves  of genus two and with a marked Weierstrass point 
is  $\mb{P}^{2, 3, 4, 5} \backslash \{\Delta\cdot t_5 = 0\}$. 
The hyperelliptic curve together with its Weierstrass point over the point $t =[t_2:t_3:t_4:t_5]$, is given by 
\[
(y^{2} = x^5+t_2x^3 +t_3x^2 +t_4x +t_5, [0:1:0]).
\]
\end{prop}
It seems natural to choose the following modular coordinates on $\ml{M}_2(w)$:
\begin{equation}
\label{6oct2016}
 j_2:=\frac{t_2^ 5}{t_5^2},\ j_3:=\frac{t_3^ 5}{t_5^3},\ j_4:=\frac{t_4^5}{t_5^3}.
\end{equation}

\subsection{Moduli space II}
In this section we define:
\begin{equation}
\label{T-variables}
T_4:=t_2,\  \ T_8:=t_4,\ \  T_{12}:=t_3^2,\ \  T_{16}:=t_3t_5,\ \  T_{20}:=t_5^2.
\end{equation}
We can write the discriminant $\Delta$ in \eqref{23.02.2019} as a polynomial in $T_i$'s and for
simplicity we denote it again by $\Delta$; being clear in the context whether it depends on $t$ or $T$. For a curve $Y$ of genus $2$ the abstract wedge product 
$H^2_\dR(J(Y)):=\wedge H^1_\dR(Y)$ has a one
dimensional subspace $F^2H_{dR}^2(J(Y))$ which is generated by the wedge product of holomorphic $1$-forms in \
$Y$.  
\begin{prop} \label{moduli sp for Siegel}
The moduli space $\Sf$ of triples $(Y, P, \omega)$, where $(Y, P)$ is as before and 
$\omega \in F^2H_{dR}^2(J(Y))$,  is 
\[
\Sf = \spec(\mb{C}[T_4,T_8,T_{12},T_{16},T_{20},  \frac{1}{T_{20}\Delta} ]/\langle T_{16}^2-T_{12}T_{20}\rangle ).
\]
The corresponding triple over a point is given by:
\[
(y^2 = x^5 + t_2x^3 + t_3 x^2 + t_4 x + t_5, [0: 1: 0], \left[\frac{\mr{d}x}{y}\right] \wedge \left [\frac{x\mr{d}x}{y}\right]).
\]
and so, we do not have a universal family over $\Sf$.
\end{prop}
\begin{proof}
For  a triple $(Y, P, \omega)$, by Proposition \ref{23feb2019} we can assume that $(Y, P)$ has the form
\[
(y^2 = x^5 + t_2x^3 + t_3 x^2 + t_4 x + t_5, [0: 1: 0]).
\]
Because of $\dim F^2H_{dR}^2(J(Y)) = 1$, $\omega = k [\frac{\mr{d}x}{y}] \wedge [\frac{x\mr{d}x}{y}]$ for some $k = \mb{C}^*$. Under the action of $a \in \mb{G}_m$, we have:
\[
(t_2, t_3, t_4, t_5, k \left[\frac{\mr{d}x}{y}\right] \wedge \left[\frac{x\mr{d}x}{y}\right]) \to (a^{-4}t_2, a^{-6}t_3, a^{-8}t_4, a^{-10}t_5, a^{-4}k \left[\frac{\mr{d}x}{y}\right] \wedge \left[\frac{x\mr{d}x}{y}\right]).
\]
Hence two triples $(Y_1, P_1, \omega_1)$ with $t$-coordinates $(t_2, t_3, t_4, t_5)$ and $(Y_2, P_2, \omega_2)$ with $t$-coordinates $(\tilde{t}_2, \tilde{t}_3, \tilde{t}_4, \tilde{t}_5)$ are isomorphic if and only if their $t$-coordinates are the same,i.e. $\tilde{t}_2 = t_2, \tilde{t}_3 = - t_3, \tilde{t}_4 =t_4, \tilde{t}_5 = -t_5$. Therefore, the moduli space of $(Y, P, \omega)$ is isomorphic to 
$\spec(\mb{C}[t_2, t_3^2, t_4, t_5^2, t_3 t_5, \frac{1}{t_5^2 \Delta} ]).$
\end{proof}

\subsection{Moduli space III}
Now let us consider the moduli of
\[
(Y, P, \alpha_1, \alpha_2, \alpha_3, \alpha_4), 
\]
where $(Y, P)$ are as before, and $(\alpha_1, \alpha_2, \alpha_3, \alpha_4)$ is a basis of $H^1_{dR}(Y)$ such that
\begin{itemize}
\item
$\alpha_1, \alpha_2 \in F^1H^1_{dR}(Y)$;
\item
the intersection form in the basis $(\alpha_i)$ is:
\begin{equation}
\label{05oct2019}
\Phi = [\langle \alpha_i, \alpha_j \rangle] = \left(\begin{matrix}
0  & 0  & 1  & 0\\
0  & 0  & 0  & 1\\
-1  & 0 & 0  & 0\\
0 & -1  & 0  & 0
\end{matrix}\right).
\end{equation}.
\end{itemize}
We recall that $\langle \alpha_i, \alpha_j \rangle:=\mbf{Tr}(\alpha_i \cup \alpha_j)= 
\frac{1}{2\pi i} \int_{Y(\mb{C})} \alpha_i \cup \alpha_j$. 
In order to construct the moduli space $\Tf$, we take a $4\times 4$ matrix $S=\mat{S_1}{0}{S_3}{S_4}$ 
with unknown entries and write $\alpha=S \omega$:
\begin{equation} \label{S-matrix}
\left(\begin{matrix}
\alpha_1 \\
\alpha_2 \\
\alpha_3 \\
\alpha_4
\end{matrix}\right) = \left( \begin{matrix}
s_{11} & s_{12} & 0 & 0 \\
s_{21} & s_{22} & 0 & 0 \\
s_{31} & s_{32} & s_{33} & s_{34} \\
s_{41} & s_{42} & s_{43} & s_{44} \\
\end{matrix} \right) \left(\begin{matrix}
\omega_1 \\
\omega_2 \\
\omega_3 \\
\omega_4
\end{matrix}\right).
\end{equation}
The constancy of the cup product in $\alpha_i$'s implies that  
$\Phi = S \Omega S^{\mr{tr}}$, where $S^{\mr{tr}}$ is the transpose matrix of $S$. 
Under the action $a \in \mb{G}_m$ on $Y_t$, for $i = 0, 1, 2, 3$, the pull back of $\omega_i$ is $a^{2i - 3} \omega_i$ (on $Y_{t \bullet a}$). Hence, under this action we have the identification:
\[
(Y_t, [0:1:0]\ ,\  \alpha) \cong (Y_{t \bullet a},  \ [0:1:0], \ S \cdot \diag(a^{-3}, a^{-1}, a, a^3)\cdot  \omega),
\]
where $\diag(a^{-3}, a^{-1}, a, a^3)$ is a diagonal matrix. Therefore,  we get that the moduli space $\Tf$ 
is isomorphic to 
\[
\left \{(t_2, t_3, t_4, t_5, S)\in\C^{16} | \Phi = S \Omega S^{\mr{tr}}, t_5\Delta \neq 0, \det S \neq 0 \right \} / 
\mb{G}_m.
\]
\begin{prop}
\label{lopes2019}
We have 
\begin{eqnarray*}
 \Tf &=& \mr{Proj}\left ( 
 \frac{
 \C\left[t_2,t_3,t_4,t_5,s_{11},s_{21},s_{31},s_{41}, s_{12},s_{22},s_{32},s_{42},\ 
 \frac{1}{t_5\Delta\cdot (s_{11}s_{22}-s_{12}s_{21})}\right ]
 }
 {
\langle   s_{42}s_{21} - s_{41}s_{22} + s_{32}s_{11} - s_{31}s_{12} - \frac{t_2}{4}    \rangle }\right)
\\
& &
\subset \P^{4,6,8,10,3,3,3,3,1,1,1,1}, 
\end{eqnarray*}
where we have considered the degrees $\deg(t_i)=2i,\ \deg(s_{i1})=3$ and $\deg(s_{i2})=1$. 
In particular, $\Tf$ is of dimension $10$.
\end{prop}
\begin{proof}
From the equation $\Phi = S \Omega S^{\mr{tr}}$, we get:
\begin{equation} \label{matrixcon1}
S_1 \Omega_2 S_4^{\mr{tr}} =  \left( \begin{matrix}
1 & 0 \\
0 & 1
\end{matrix} \right) 
\end{equation}
and 
\begin{equation} \label{matrixcon2}
S_4\Omega_3 S_3^{\mr{tr}} + S_3\Omega_2 S_4^{\mr{tr}} + S_4\Omega_4 S_4^{\mr{tr}} = 0.
\end{equation}
Using (\ref{matrixcon1}), we get identities:
$S_4 \Omega_3   = -(S_1^{\mr{tr}})^{-1}$ and $\Omega_2 S_4^{\mr{tr}} = S_1^{-1}$. Putting them in the equation (\ref{matrixcon2}), we have an equation:
\begin{equation}
- (S_1^{\mr{tr}})^{-1}S_3^{\mr{tr}} + S_3 S_1^{-1}  + S_4 \Omega_4 S_4^{\mr{tr}} = 0,
\end{equation}
which is equivalent to the equation
\begin{equation} \label{s-eq5}
s_{42}s_{21} - s_{41}s_{22} + s_{32}s_{11} - s_{31}s_{12} - \frac{t_2}{4} = 0.
\end{equation}
\end{proof}
In Proposition \ref{lopes2019} we can discard the variable $t_2$ as this is equal to $4(s_{42}s_{21} - s_{41}s_{22} + s_{32}s_{11} - s_{31}s_{12})$.  In this way we obtain the fact that $\Tf$ is an open subset of the weighted projective space 
$\P^{6,8,10,3,3,3,3,1,1,1,1}$ given by $t_5\cdot \Delta\cdot (s_{11}s_{22}-s_{12}s_{21})\not=0$. This appears in Theorem \ref{maintheo}, part \ref{mt3}.

\subsection{Ring of functions of $\Tf$} 
\label{affinechart}
In this section we find the smallest $N$ such that we can realize $\Tf$ as an affine subvariety of $\C^N$. 
\begin{prop}
\label{14.04.2019.dubai}
Let $\D = \det S_1=s_{11}s_{22}-s_{12}s_{21}$. There is an embedding from $\Tf \to \mb{C}^{153}$, which is given by 
sending
$(t_2,t_3,t_4,t_5,s_{11},s_{21},s_{31},s_{41}, s_{12},s_{22},s_{32},s_{42})$ to 
\begin{equation} \label{newvariables}
\begin{split}
& T_4 = t_2 / \D,  \ \ \ T_8 = t_4 / \D^2,  \ \ \ T_{12} = t^2_3 / \D^3,  \ \ \  T_{16} = t_3t_5 / \D^4,  \ \ \ T_{20} = t^2_5 / \D^5,   \\
& Q_{i_1i_2} = \frac{s_{i_11}s_{i_22}}{\delta}, 1 \leq i_1, i_2 \leq 4, \\
& Q_{i_1i_2i_3i_4} = \frac{(s_{12})^{i_1}(s_{22})^{i_2}(s_{32})^{i_3}(s_{42})^{i_4}}{\delta}, 0 \leq i_1, i_2, i_3, i_4 \leq 4, \sum_{j = 1}^4 i_j = 4, \\
&P_{i_1i_2i_3i_4} = \frac{(s_{11})^{i_1}(s_{21})^{i_2}(s_{31})^{i_3}(s_{41})^{i_4}}{\delta^3}, 0 \leq i_1, i_2, i_3, i_4 \leq 4,  \sum_{j = 1}^4 i_j = 4, \\
& U_{3i_1i_2} = \frac{t_3 s_{i_12}s_{i_2 2}}{\delta^2}, U_{5i_1i_2} = \frac{t_5 s_{i_12}s_{i_22}}{\delta^3}, 1 \leq i_1, i_2 \leq 4, \\
& V_{3i_1i_2} = \frac{t_3 s_{i_11}s_{i_2 1}}{\delta^3}, V_{5i_1i_2} = \frac{t_5 s_{i_11}s_{i_21}}{\delta^4}, 1 \leq i_1, i_2 \leq 4. \\
\end{split}
\end{equation}
\end{prop}
Note that via the equation (\ref{s-eq5}), we have: 
\[T_4 = 4(Q_{24} - Q_{42} + Q_{13} - Q_{31})\] and 
\[
Q_{21} - Q_{12}  =1,
\]
and so we do not need two functions in \eqref{newvariables} which consists of $153+2$ functions.  
\begin{proof}
Let $U$ be the open affine subvariety of $\mb{P}^{4,6,8,10, 3,3,3,3,1,1,1,1}$, 
which is the complement of $\delta = 0$. By the definition of $\Tf$, there is an embedding 
$\Tf \to U$. We just want to find enough invariants under the action of $\mb{G}_m$ generating the coordinate functions on $U$. According to the degrees of variables $t_m, s_{ij}$, 
if $\delta^{-n}\prod^{5}_{m = 2}t^{n_m}_m \prod_{1 \leq i \leq 4, 1 \leq j \leq 2} s_{ij}^{n_{ij}}$ is a function on $U$, then we need to make sure that:
\begin{equation} \label{generator degfun}
4n = 4n_2 + 6n_3 + 8n_4 + 10n_5 + 3 \sum^4_{i = 1} n_{i1} + \sum^4_{i = 1} n_{i2}.
\end{equation}
Moreover, for generators of the ring of regular functions on $U$, we can assume that
$0 \leq n_m \leq 1, 0 \leq n_{ij} \leq 2$ for $2 \leq m \leq 5$ and $1 \leq i \leq 4, 1 \leq j \leq 2$. For the solutions of (\ref{generator degfun}), we may divide into the following types:
\begin{itemize}
\item $4n = 4n_2 + 6n_3 + 8n_4 + 10n_5$;
\item $4n = 3 \sum^4_{i = 1} n_{i1} + \sum^4_{i = 1} n_{i2}$, where $n_{ij} \neq 0$;
\item $4n = \sum^4_{i = 1} n_{i2}$;
\item $4n = 3\sum^4_{i = 1} n_{i2}$;
\item $4n = 6n_3 + 10n_5 + \sum^4_{i = 1} n_{i2}$;
\item $4n = 6n_3 + 10n_5 + 3 \sum^4_{i = 1} n_{i1}$.
\end{itemize}
In each case, we will get the solutions as stated in the proposition respectively.
\end{proof}

\subsection{Algebraic group $\BG$} \label{alggp}
The algebraic group $\BG$
$$
\BG= \left\{
\left(\begin{matrix}
k & k^{'} \\
0 & k^{-\mr{tr}}
\end{matrix} \right) 
\in GL(4, \mb{C}) \Bigg| k(k^{'})^{\mr{tr}}=(k(k^{'})^{\mr{tr}})^{\mr{tr}} \right\} \subset \mr{Sp}(4, \mb{C})
$$
acts from the left on $\Tf$ by the base change:
\[
\Tf \times \BG \to  \Tf, t =((Y, [0:1:0], \alpha), \mbf{g}) \to t \bullet \mbf{g} = (Y_t, [0:1:0], 
\alpha\cdot \mbf{g}),
\]
where we regard $\alpha$ as a $1\times 4$ matrix and  $\alpha\cdot \mbf{g}$ is the usual multiplication of matrices. 
The algebraic group $\BG$ is also called a Siegel parabolic subgroup of 
$\mr{Sp}(4,\mb{C})$. 
For $\mbf{g} = \begin{bmatrix}
k & k^{'} \\
0 & k^{-\mr{tr}} \end{bmatrix}$, we have  
$\delta(t \bullet \mbf{g}) = \det(k) \cdot \delta(t)$. From this we get 
the functional equations for $T_i$:
\begin{equation}
\label{14april2019}
T_i(t \bullet \mbf{g}) = (\det k)^{-\frac{i}{4}}T_i(t),\ \ t\in\Tf. 
\end{equation}
The functional equations of other regular functions on $\Tf$ are not as short as \eqref{14april2019} and we do not reproduce them here.

\section{Modular vector fields}
In this section we describe three vector fields on $\Tf$ which are algebraic incarnation of the derivations/vector fields  $\frac{\partial}{\partial \tau_i},\ \ i=1,2,3$. 
\label{s4}
\begin{theorem}
There are unique global vector fields $\Ra_k,\ \ 1 \leq  k \leq 3$ on $\Tf$ such that
\begin{equation}  \label{vector field}
\nabla_{\Ra_k} \left(\begin{matrix}
\alpha_1 \\
\alpha_2 \\
\alpha_3 \\
\alpha_4
\end{matrix}\right) = \left(\begin{matrix}
0 & C_{k} \\
0 & 0
\end{matrix}\right)\left(\begin{matrix}
\alpha_1 \\
\alpha_2 \\
\alpha_3 \\
\alpha_4
\end{matrix}\right),
\end{equation}
where $C_{1} = \left(\begin{matrix}
1 & 0 \\
0 & 0
\end{matrix}\right), C_{2} = \left(\begin{matrix}
0 & 0 \\
0 & 1
\end{matrix}\right), C_{3} =\left(\begin{matrix}
0 & 1 \\
1 & 0
\end{matrix}\right)$. Here, $\nabla$ is the Gauss-Manin connection of the family of genus two curves over $\Tf$. 
\end{theorem}
\begin{proof}
Our proof actually  computes $\Ra_k,\ \ k=1,2,3$ and it is as follows. For the relevant computer code see Appendix \ref{03o2019}. Let us first consider a vector field 
\[
\Ra_k = \sum_{m = 2}^5 \Ra_{k, m} \frac{\partial}{\partial t_m} + \sum_{1 \leq i \leq 4, 1 \leq j \leq 2} \Ra_{k, ij} \frac{\partial}{\partial s_{ij}} ,
\]
with unknown coefficients $\Ra_{k,m}$ and $\Ra_{k,ij}$. We have 
\begin{equation} \nonumber
\nabla_{\Ra_k}  \alpha = \nabla_{\Ra_k} (S \omega)  = (\nabla_{\Ra_k} S) \omega + S \A(\Ra_k) \omega = (\mr{d}S(\Ra_k) + 
S\A(\Ra_k))S^{-1}\alpha.
\end{equation}
The equation (\ref{vector field}) is equivalent to
\begin{equation} \label{vector field equ}
\mr{d}S(\Ra_k) + S\A(\Ra_k) =\check C_k S,
\end{equation} 
where $\check C_k = \left(\begin{matrix}
0 & C_k \\
0 & 0
\end{matrix}\right)$.
Recall that we use the notation
$S = \left( \begin{matrix}
S_1 & 0 \\
S_3 & S_4
\end{matrix} \right)$ and
$B_m = \left( \begin{matrix}
B_{m, 1} & B_{m, 2} \\
B_{m, 3} & B_{m, 4}
\end{matrix} \right)$ for $2 \leq m \leq 5$.
From the equation (\ref{vector field equ}), we get an equation:
\begin{equation} \label{matrix equ for vector field}
\left(\begin{matrix}
\Ra_k(S_1) & 0 \\
\Ra_k(S_3) & \Ra_k(S_4)
\end{matrix}\right) + \left( \begin{matrix}
S_1 & 0 \\
S_3 & S_4
\end{matrix} \right) \sum^5_{m=2} \left( \begin{matrix}
B_{m, 1} & B_{m, 2} \\
B_{m, 3} & B_{m, 4}
\end{matrix} \right)\Ra_{k,m} = \left(\begin{matrix}
0 & C_k \\
0 & 0
\end{matrix}\right)\left( \begin{matrix}
S_1 & 0 \\
S_3 & S_4
\end{matrix} \right).
\end{equation}
Considering the right-upper $2 \times 2$ matrices, we get its equivalent form:
\begin{equation}
\sum^5_{m=2} \Ra_{k,m} \cdot B_{m ,2} = S_1^{-1} C_k S_4.
\end{equation}
From the equation (\ref{matrixcon1}), the right-hand side of the above equation can be transformed into:
\begin{equation}
\begin{split}
S_1^{-1} C_k S_4 & = S_1^{-1} C_k S_1^{-\mr{tr}} (\Omega_2)^{-\mr{tr}} \\
& = \frac{1}{(\det S_1)^2}\left( \begin{matrix}
s_{22}^2 & -s_{21}s_{22} \\
-s_{22}s_{21} & s_{21}^2
\end{matrix} \right) (\Omega_2)^{-\mr{tr}} \\
& = \frac{1}{(\det(S_1))^2} \left( \begin{matrix}
-\frac{1}{4}s_{21}s_{22} & \frac{3}{4} s_{22}^2 \\
\frac{1}{4}s_{21}^2 & - \frac{3}{4} s_{21}s_{22}
\end{matrix} \right),
\end{split}
\end{equation}
for $k=1$. Similarly we do the computation for $k=2,3$ and define: 
\begin{equation}
\begin{bmatrix}
s_{1, -4} \\
s_{1, -6} \\
s_{1, -2} \\
\end{bmatrix}=
\frac{1}{(\det(S_1))^2} \left( \begin{matrix}
&-\frac{1}{4}s_{21}s_{22} \\
& \frac{3}{4} s_{22}^2 \\
& \frac{1}{4}s_{21}^2 
\end{matrix} \right),
\ \ \ 
\begin{bmatrix}
s_{2, -4} \\
s_{2, -6} \\
s_{2, -2} 
\end{bmatrix}=
\frac{1}{(\det(S_1))^2} \left( \begin{matrix}
& -\frac{1}{4} s_{11}s_{12} \\
& \frac{3}{4} s_{12}^2  \\
& \frac{1}{4} s_{11}^2 
\end{matrix} \right),
\end{equation}
and
\begin{equation}
\begin{bmatrix}
s_{3, -4} \\
s_{3, -6} \\
s_{3, -2} 
\end{bmatrix} =
\frac{1}{(\det(S_1))^2} \left( \begin{matrix}
&\frac{1}{4}(s_{12}s_{21}+s_{11}s_{22} )\\
& - \frac{3}{2} s_{12}s_{22} \\
& - \frac{1}{2}s_{11}s_{21} \\
\end{matrix} \right).
\end{equation}
The existence of $\Ra_k$ depends on the solution of 
\begin{equation} \label{gamma-matrix}
\left(\begin{matrix}
b^2_{13} & b^3_{13} & b^4_{13} & b^5_{13} \\
b^2_{14} & b^3_{14} & b^4_{14} & b^5_{14} \\
b^2_{23} & b^3_{23} & b^4_{23} & b^5_{23} \\
b^2_{24} & b^3_{24} & b^4_{24} & b^5_{24} 
\end{matrix}\right) \left(\begin{matrix}
\Ra_{k,2} \\
\Ra_{k,3} \\
\Ra_{k,4} \\
\Ra_{k,5}
\end{matrix}\right) = 
\left(\begin{matrix}
s_{k, -4} \\
s_{k, -6} \\
s_{k, -2} \\
3s_{k, -4}
\end{matrix}\right).
\end{equation}
The vector $[2t_2,3t_3,4t_4,5t_5]^{\tr}$ is in the kernel of the above 
$4\times 4$ matrix.
This also follows from the fact that  the Euler vector field $\sum_{m=2}^5 mt_m \frac{\partial}{\partial t_m}$ induces the zero vector field in the weighted projective space $\P^{4,6,8,10,3,3,3,3,1,1,1,1}$.  The solution of (\ref{gamma-matrix}) up to addition of the Euler vector field is unique and it
is given by
\begin{equation} \nonumber
\begin{split}
\Ra_{k, 2} = &\frac{1}{75t_5}
((24t_{2}^{2}t_{4}+450t_{3}t_{5})s_{k, -2} \\+ &(-180t_{2}^{2}t_{5}+ 36t_{2}t_{3}t_{4} + 600t_{4}t_{5} ) s_{k, -4} 
 +  (-40t_{2}t_{3}t_{5}+  16t_{2}t_{4}^{2} +250t_{5}^{2})s_{k, -6}),
\end{split}
\end{equation}
\begin{equation} \nonumber
\begin{split}
\Ra_{k, 3} = &\frac{1}{75t_5}
((-180t_{2}^{2}t_{5}+36t_{2}t_{3}t_{4}+600t_{4}t_{5})s_{k, -2} \\
+ &(-390t_{2}t_{3}t_{5} + 54t_{3}^{2}t_{4}+750t_{5}^{2})s_{k, -4} 
+ (-20t_{2}t_{4}t_{5}-60t_{3}^{2}t_{5}+24t_{3}t_{4}^{2})s_{k, -6}),
\end{split}
\end{equation}
\begin{equation} \nonumber
\begin{split}
\Ra_{k, 4} = &\frac{1}{75t_5}
((-120t_{2}t_{3}t_{5}+48t_{2}t_{4}^{2}+750t_{5}^{2})s_{k, -2} \\
+&(-60t_{2}t_{4}t_{5}-180t_{3}^{2}t_{5}+72t_{3}t_{4}^{2})s_{k, -4} 
+(150t_{2}t_{5}^{2}-110t_{3}t_{4}t_{5} + 32t_{4}^{3})s_{k, -6}),
\end{split}
\end{equation}
\begin{equation} \nonumber
\Ra_{k, 5} =  0.
\end{equation}
 Despite the fact that the Gauss-Manin connection matrix $B$ is not simple, we can compute 
 the matrices $E_{-2}, E_{-4}, E_{-6}$ through the equality
\[
E(k):=\sum^5_{m = 2} \Ra_{k, m} B_m =  s_{k, -2} E_{-2} + s_{k, -4} E_{-4} + s_{k, -6} E_{-6} 
\]
for $k = 1, 2, 3$,
and they have rather simple expressions:
$$
E_{-2} = \left(
\begin{array}{*{4}{c}}
-\frac{6t_{2}t_{4}}{25t_5} & -1 & 0 & 0 \\
\frac{4}{5} t_{2} & -\frac{2t_{2}t_{4}}{25t_5} & 1 & 0 \\
0 & \frac{8}{5} t_{2} & \frac{2t_{2}t_{4}}{25t_5} & 3 \\
-t_{4} & -2t_{3} & -\frac{3}{5}t_{2} & \frac{6t_{2}t_{4}}{25t_5}
\end{array}
\right),
$$
$$
E_{-4} = \left(
\begin{array}{*{4}{c}}
(\frac{4}{5}t_{2}-\frac{9t_{3}t_{4}}{25t_5}) & 0 & 1 & 0 \\
\frac{6}{5}t_3 & (\frac{8}{5}t_{2}-\frac{3t_{3}t_{4}}{25t_5}) & 0 & 3 \\
-t_{4} & \frac{2}{5}t_{3} & -(\frac{3}{5}t_{2}-\frac{3t_{3}t_{4}}{25t_5}) & 0 \\
-2t_{5} & -3t_{4} & -\frac{2}{5}t_{3} & -(\frac{9}{5}t_{2} -\frac{9t_{3}t_{4}}{25t_5})
\end{array}
\right),
$$
$$
E_{-6} = \left(
\begin{array}{*{4}{c}}
(\frac{2}{5}t_{3}-\frac{4t_{4}^{2}}{25t_5}) & \frac{1}{3}t_{2} & 0 & 1 \\
\frac{1}{5}t_{4} & (\frac{2}{15}t_{3}-\frac{4t_{4}^{2}}{75t_5}) & 0 & 0 \\
-\frac{2}{3}t_{5} & \frac{1}{15}t_{4} & -(\frac{2}{15}t_{3}-\frac{4t_{4}^{2}}{75t_5}) & 0 \\
0 & -\frac{4}{3}t_{5} & -\frac{1}{15}t_{4} & -(\frac{2}{5}t_{3}-\frac{4t_{4}^{2}}{25t_5})
\end{array}
\right).
$$
Putting these solutions back to (\ref{matrix equ for vector field}), we get $\Ra_{k, ij}$:
\begin{eqnarray*}
\begin{bmatrix}
dS_1(\Ra_k) \\
dS_3(\Ra_k)
\end{bmatrix}
&=&  
\begin{bmatrix}
C_k S_3 \\
0
\end{bmatrix} - \begin{bmatrix}
S_1(\sum_{m = 2}^5 \Ra_{k, m} B_{m,1}) \\
S_3(\sum_{m = 2}^5 \Ra_{k, m} B_{m,1}) + S_4(\sum_{m = 2}^5 \Ra_{k, m} B_{m,3}) 
\end{bmatrix} 
\\
&= &
\begin{bmatrix}
 C_kS_3-S_1E(k)_1 \\
 -S_3E(k)_1-S_4E(k)_3
\end{bmatrix},
\end{eqnarray*}
where we have used our convention $E(k)=\begin{bmatrix}
E(k)_{1} & E(k)_{2} \\
E(k)_{3} & E(k)_{4}
\end{bmatrix}$. 
It remains to prove that $\Ra_k$ is tangent to the loci $F=0$, where $F$ is the polynomial \eqref{s-eq5}.
This follows from the equations:
\begin{equation} \nonumber
\frac{dF(\Ra_1)}{F}= \frac{1}{25t_5\delta^2}(2t_2t_4s_{21}^2+15t_2t_5s_{21}s_{22}-3t_3t_4s_{21}s_{22}-10t_3t_5s_{22}^2+4t_4^2s_{22}^2),
\end{equation} 
\begin{equation} \nonumber
\frac{dF(\Ra_2)}{F}= \frac{1}{25t_5\delta^2}(2t_2t_4s_{11}^2+15t_2t_5s_{11}s_{12}-3t_3t_4s_{11}s_{12}-10t_3t_5s_{12}^2+4t_4^2s_{12}^2),
\end{equation} 
\begin{equation} \nonumber
\begin{split}
\frac{dF(\Ra_3)}{F}= & \frac{1}{25t_5\delta^2}(-4t_2t_4s_{11}s_{21}-15t_2t_5s_{11}s_{22}-15t_2t_5s_{12}s_{21} + 3t_3t_4s_{11}s_{22}  + \\ &3t_3t_4s_{12}s_{21}+20t_3t_5s_{12}s_{22}-8t_4^2s_{12}s_{22}).
\end{split}
\end{equation} 
\end{proof}

All $\Ra_k$'s are tangent to $\Delta=0$ and this follows from the computation:
\begin{equation}
\frac{\mr{d} \Delta}{\Delta}(\Ra_k) = 
\frac{16t_{2}t_{4}}{5t_{5}}\cdot s_{k, -2} - \frac{60t_{2}t_{5}-24t_{3}t_{4}}{5t_{5}}\cdot s_{k, -4}-\frac{60t_{3}t_{5}-32t_{4}^{2}}{15t_{5}}\cdot s_{k, -6}.
\end{equation}


Let us define 
$$
H_k:= S_1^{-1}C_k S_3,\ \ \ 
J_k:= S_1^{-1} C_k S_4=
\begin{bmatrix}
s_{k,-4} & s_{k,-6}\\ s_{k,-2} & 3s_{k,-4} 
\end{bmatrix}
,\ \ \ k = 1, 2, 3.
$$
The following proposition might be useful for a better understanding of $\Ra_k$'s.  
\begin{prop}
We have the following equations:
\begin{eqnarray*}
\Ra_{\check{k}}(J_k) &=&  H_{\check{k}}J_k -E(\check{k})_1J_k - H_k E(\check{k})_2 - J_k E(\check{k})_4,
 \\
 \Ra_{\check{k}}(H_k) &=& 
 H_{\check{k}}H_k - E(\check{k})_1H_k - H_k E(\check{k})_1 - J_k E(\check{k})_3,
\end{eqnarray*}
for $\check{k} = 1, 2, 3.$
\end{prop}
\begin{proof}
We only check the first equation. Similarly, one can verify the other. 
\begin{eqnarray*}
\Ra_{\check{k}}(J_k) &=& 
\Ra_{\check{k}}( S_1^{-1} C_k S_4) = \Ra_{\check{k}}(S^{-1}_1)C_kS_4 + S^{-1}_1 C_k \Ra_{\check{k}}(S_4) 
\\
&=& 
S_1^{-1} \Ra_{\check{k}}(S_1) S_1^{-1} C_k S_4 + S^{-1}_1 C_k \Ra_{\check{k}}(S_4) 
\\
&=& 
S_1^{-1}(C_{\check{k}}S_3 - \sum_m \Ra_{\check{k}, m}S_1B_{m, 1})S_1^{-1}C_kS_4 + 
\\
& &
S_1^{-1}C_k(- \sum_m \Ra_{\check{k}, m}(S_3 B_{m, 2} + S_4B_{m, 4})) 
\\
& \stackrel{\eqref{matrix equ for vector field}}{=}& 
S_1^{-1}C_{\check{k}}S_3J_k - \sum_m \Ra_{\check{k}, m}(B_{m, 1}J_k + (S_1^{-1}C_kS_3)B_{m, 2} + J_kB_{m, 4}) 
\\
&=& 
H_{\check{k}}J_k - \sum_m \Ra_{\check{k}, m}(B_{m, 1}J_k + H_kB_{m, 2} + J_kB_{m, 4}) 
\\
&=&
 H_{\check{k}}J_k -E(\check{k})_1J_k - H_k E(\check{k})_2 - J_k E(\check{k})_4.
\end{eqnarray*}
\end{proof}

\subsection{The description of $\Ra_k$'s in an affine chart}
\label{JH2019}
Let $f$ be a homogeneous polynomial in $\mb{C}[x_0, \cdots, x_n]$ with $\deg x_i = \alpha_i\in\N$ and $\P^\alpha:=\mr{Proj}(\mb{C}[x_0, \cdots, x_n])$ be the
weighted projective space.  The coordinate ring of $U:=\P^\alpha-\{f=0\}$ is given by the polynomial ring generated by $\frac{x_0^{a_0} \cdots x_n^{a_n} }{f^k}$
such that $\sum_{i=0}^n a_i \alpha_i = k \deg(f)$.
We find generators $X = (X_i)_{i = 1, \cdots, m}$ of the coordinate ring of 
$U$ and we have $\ml{O}(U) ={\spec}(\mb{C}[X_1, \cdots, X_m] / I)$, where $I$ is the ideal generated by the polynomials $P$ with variables in $X$ such that $P(X) = 0$. 
A vector field on $\mb{P}^{\alpha}$ is given by 
$v=\sum_{i = 0}^n v_i(x)\frac{\partial}{\partial x_i}$, 
where $v_i(x) = \frac{f_i}{g_i}$ can be written as the ratio of two homogeneous polynomials $f_i$ and $g_i$ such that $\deg f_i - \deg g_i = \alpha_i$.
Then we can write $\mr{d}X_i (v)$ in terms of $X_i$, denoted by $Q_i(X)$, and then we have:
\begin{equation}
v|_{U} = \sum_{i = 1}^m Q_i \frac{\partial}{\partial X_i}.
\end{equation}
This can be used to describe the vector fields $\Ra_k$'s in
the complement $U$ of $\delta = 0$ in $\mb{P}^{4,6,8,10,3,3,3,3,1,1,1,1}$. Note that we have already chosen coordinate functions for $U$ in \S\ref{affinechart}.

\section{The monodromy group of genus 2 hyperelliptic curves} 
\label{s6}
In this section, we want to compute the monodromy group of 
the moduli space of genus $2$ hyperelliptic curves with a marked Weierstrass point. 
\subsection{Picard-Lefschetz theory}
First we recall some general statements. We denote the moduli space of hyperelliptic curves of genus 
$g$ by $\ml{C}_g$. For each hyperelliptic curve $g$, we can associate it with a polynomial of degree $d$ with a non-vanishing discriminant such that the affine part of $Y$ is given by $y^2 = P(x)$. Then the genus of $Y$ is $g= [\frac{d-1}{2}]$. We fix a base point $b=0 \in \ml{C}_g$, which is corresponding to the hyperelliptic curve 
whose affine part is given by 
\begin{equation}
\label{12apr2019}
Y_b: y^2 = x^d+1.  
\end{equation}
We denote the monodromy map:
\[
\pi_1(\ml{C}_g, b) \to Aut(H_1(Y_b, \mb{Z}), \cdot) \cong \mr{Sp}(2g, \mb{Z})
\]
by $h$, where $\cdot$ denotes the intersection pairing.
\

We define $f: \mb{C}^2 \to \mb{C}$ by $f(x, y) = - y^2 + P(x)$. Then we let $\{p_1,p_2, \cdots, p_{d-1} \}$ be the critical points of $f$, i.e., $p_i = (a_i, 0)$ such that $P^{'}(a_i) = 0$. Moreover, we denote the critical values of $f$ by $\{c_1, c_2, \cdots, c_{d-1}\}$. One may choose a set of vanishing cycles 
$\delta_1,\delta_2, \cdots, \delta_{d-1}$ such that their intersection numbers are given by the Dynkin diagram of the polynomial $P(x)$, i.e.
\begin{equation}
\begin{split}
&\delta_1 \cdot \delta_2 =  - \delta_2 \cdot \delta_1 = \delta_2 \cdot \delta_3 = - \delta_3 \cdot \delta_2 =  \cdots \delta_{d-2} \cdot  \delta_{d-1} = - \delta_{d-1} \cdot  \delta_{d-2} = 1 \\
& \mathrm{otherwise} \ \ \ \delta_i \cdot \delta_j = 0,
\end{split}
\end{equation}
see for instance \cite[Chapter 7]{ho13}. 
Moreover if $d$ is odd, we can use linear combinations of $\delta_i$ to form a symplectic basis $e_i$ of $H_1(Y_b, \mb{Z})$ :
\begin{equation}
\label{12apr2019-2}
\begin{split}
& e_1 = \delta_1, e_2 = \delta_3, \cdots, e_{\frac{d-1}{2}} = \delta_{d-2}, \\
& e_{\frac{d+1}{2}} = \delta_2+\delta_4, e_{\frac{d+3}{2}} = \delta_4+\delta_6, \cdots, e_{d-2} = \delta_{d-3}+\delta_{d-1}, e_{d-1} = \delta_{d-1}.
\end{split}
\end{equation}
If $d$ is even, the similar $e_i$ consists of a symplectic basis of $H_1(Y_b, \mb{Z})$ and $\delta_1 + \delta_3 + \cdots + \delta_{d-1} =0$ in $H_1(Y_b, \mb{Z})$.
\

We recall the Picard-Lefschetz formula (see for instance in \cite[\S 6.6]{ho13}). Suppose that $f$ is a holomorphic map from a projective complex manifold $Y$ of dimension $n$ to the projective line $\mb{P}^1$. Also suppose that all critical points are non-degenerate (not necessarily in different fibers). We let $b$ be a regular value of $f$. Then the monondromy $h$ around the critical value $c_i$ is given by the Picard-Lefschetz formula
\begin{equation} \label{PL}
h(\delta) = \delta - \sum_j\langle \delta, \delta_j \rangle \delta_j, \ \ \delta \in H_{1}(Y_b),
\end{equation}
where $j$ runs through all the Lefschetz vanishing cycles in the singularities of $Y_{c_i}$ and $\langle \cdot, \cdot \rangle$ denotes the intersection number of two cycles in $Y_b$.

\subsection{Monodromy of genus two curves} 
\label{mono}
Now we assume $d= 5$. Then there are $4$ critical points $c_i$. For simplicity, we choose a symplectic basis of $H_1(Y_b, \mb{Z})$:
\begin{equation} \label{symbasisforfixedcurve}
e_1 = \delta_1,e_2 = \delta_3, e_3 = \delta_2 + \delta_4, e_4 =  \delta_4.
\end{equation}
Using the Picard-Lefschetz formula and the symplectic basis $\{e_i\} \in H_1(Y_b, \mb{Z})$, we get:
$$h_{c_1}(e_i) = e_i - \langle e_i, \delta_1 \rangle \delta_1.$$
Under the basis $(e_i)$, $h_{c_1}$ corresponds to the matrix
\[\left(\begin{matrix}
1 & 0 & 0 & 0 \\
0 & 1 & 0 & 0 \\
1 & 0 & 1 & 0 \\
0 & 0 & 0 & 1 
\end{matrix}\right).\]
Similarly we will get other $3$ matrices and hence the monodromy group of genus 
$2$ hyperelliptic curves with a Weierstrass point is generated by four generators of $\Gamma$ 
in (\ref{14.02.2109}). Let us call these matrices $A,B,C$ and $D$, respectively. 
Note that if we consider the moduli space of genus 2 hyperelliptic curves, then the similar discussion implies that $\mr{Sp}(4, \mb{Z})$ is generated by the above $4$ matrices together with
\begin{equation}
E = \left(\begin{matrix}
1 & 0 & 0 & 0 \\
0 & 1 & 0 & 0 \\
1 & 1 & 1 & 0 \\
1 & 1 & 0 & 1 
\end{matrix}\right).
\end{equation}
We let $\Gamma^0_2$ be the mapping class group associated with the orbifold fundamental group of  
the moduli space of nondegenerated Riemann surfaces of genus $2$. Then it is well know that $\Gamma^0_2$ has the following presentation:
\begin{equation}
\begin{split}
\mathrm{generators}: &\zeta_1, \cdots, \zeta_5, \\
\mathrm{relations}: & [\zeta_i, \zeta_j] = 1, \mid i-j\mid > 1 \\
& \zeta_i\zeta_{i+1}\zeta_i = \zeta_{i+1} \zeta_{i} \zeta_{i+1}, \\
& (\zeta_1\zeta_2\zeta_3\zeta_4\zeta_5)^6 = 1, \zeta^2 = 1, [\zeta, \zeta_i] = 1, \\
\end{split}
\end{equation}
where $\zeta = \zeta_1\zeta_2\zeta_3\zeta_4\zeta_5^2 \zeta_4 \zeta_3 \zeta_2 \zeta_1 $, see for instance  \cite[Page 577]{1994Brownstein}. Hence we may view $\Gamma^0_2$ as a natural quotient of the braid group $B_6$, whose generators are also denoted by $\zeta_i$. It also known that the map $\Gamma^0_2 \to \mr{Sp}(4, \mb{Z})$ is surjective and the image of $\zeta_1$ (resp $\zeta_2,\zeta_3,\zeta_4,\zeta_5$) is $A$ (resp $B, C, D, E$). Therefore, $A, B, C, D, E$ generate the whole group $\mathrm{Sp}(4, \mb{Z})$.
We let $\Gamma$ be the subgroup of $\mathrm{Sp}(4, \mb{Z})$ generated by $A, B, C, D$.
Note that $E^2 = (ABC)^4 \in \Gamma$. 
We denote the congruence group of $\mr{Sp}(4, \mb{Z})$ with level two by  $\Gamma(2)$, i.e., $\Gamma(2) = \mr{Ker}(\mr{Sp}(4, \mb{Z}) \to \mr{Sp}(4, \mb{Z}/2\mb{Z}))$.
\begin{prop}
$\Gamma$ contains the congruence group $\Gamma(2)$ as a subgroup. 
\end{prop}
\begin{proof}
Note that by our definition of $A, B, C, D, E$, there is a natural map from the $6$-th braid group $B_6$ to $\mathrm{Sp}(4, \mb{Z})$:
\[
\phi: B_6 \to \mathrm{Sp}(4, \mb{Z}),
\]
given by $\phi(\zeta_1) = A, \phi(\zeta_2) = B, \phi(\zeta_3) = C, \phi(\zeta_4) = D$ and $\phi(\zeta_5) = E$, see \cite[\S 1.1]{2008Kassel}. Moreover we have the following commutative diagram:

\begin{xy}
(190,0)*+{P_6}="v1";(220,0)*+{B_6}="v2";(250,0)*+{\Sigma_6}="v3";
(190,-20)*+{\Gamma(2)}="v4";(220,-20)*+{\mathrm{Sp}(4, \mb{Z})}="v5";(250,-20)*+{\mathrm{Sp}(4, \mb{Z} / 2 \mb{Z})}="v6";
{\ar@{->} "v1";"v2"};{\ar@{->}^{\pi} "v2";"v3"};
{\ar@{->} "v4";"v5"};{\ar@{->} "v5";"v6"};
{\ar@{->} "v1";"v4"};
{\ar@{->}^{\phi} "v2";"v5"};{\ar@{->}^{\cong} "v3";"v6"};
\end{xy}
\noindent where $\pi$ is the map sending $\zeta_i$ to $s_i = (i, i+1)$ and $P_6 = \mr{Ker}(\pi)$. We want to show that $\Gamma(2) \subset \Gamma = \langle A,B,C,D \rangle$. Note that a set of generators of $P_6$ also provides a set of generators of $\Gamma(2)$. Now we can consider a set of generators of $P_6$ given by Corollary 1.19 in \cite[Page 21]{2008Kassel} and then we want to show that the corresponding generators of $\Gamma(2)$ can be expressed by $A, B, C, D$. Besides the ones lying in $\Gamma$, we need to show that:
\begin{equation} \nonumber
EDCBA^2B^{-1}C^{-1}D^{-1}E^{-1}, EDCB^2C^{-1}D^{-1}E^{-1}, EDC^2D^{-1}E^{-1}, ED^2E^{-1} \in \Gamma.
\end{equation}
By the explicit computation, we find that:
\[
ED^2E^{-1} = (ABCD)^5(AB)^3D^{-2}(ABC)^{-4} \in \Gamma,
\]
and
\[
EDC^2D^{-1}E^{-1} = (CD)^3(ABCD)^{-5}(AB)^{-3}(DA)^2 \in \Gamma.
\]
Moreover, we also have an identity:
\[
(EDCBA^2B^{-1}C^{-1}D^{-1}E^{-1})^{-1} = EDCB^2C^{-1}D^{-1}E^{-1} [A^2 (CD)^3 (ABCD)^5].
\]
This implies that we only need to show that $EDCBA^2B^{-1}C^{-1}D^{-1}E^{-1} \in \Gamma$. After a lot of computations, we get that:
\[
EDCBA^2B^{-1}C^{-1}D^{-1}E^{-1} = (DCB)^4 \in \Gamma.
\]
Therefore, we have $\Gamma(2) \subset \Gamma$.
\end{proof}
\begin{cor}
The index of $\Gamma$ in $\mr{Sp}(4, \mb{Z})$ is $6$. Moreover we have a short exact sequence
\begin{equation}
1 \to \Gamma(2) \to \Gamma \to \Sigma_5 \to 1.
\end{equation}
\end{cor}
\begin{proof}
We have the following diagram:
\

\begin{xy}
(190,0)*+{\mr{Ker}(\pi_{\Gamma})}="v1";(220,0)*+{\Gamma}="v2";(250,0)*+{\Sigma_5}="v3";
(190,-20)*+{\Gamma(2)}="v4";(220,-20)*+{\mathrm{Sp}(4, \mb{Z})}="v5";(250,-20)*+{\mathrm{Sp}(4, \mb{Z} / 2 \mb{Z}) \cong \Sigma_6}="v6";
{\ar@{->} "v1";"v2"};{\ar@{->}^{\pi_{\Gamma}} "v2";"v3"};
{\ar@{->} "v4";"v5"};{\ar@{->} "v5";"v6"};
{\ar@{->} "v1";"v4"};
{\ar@{->} "v2";"v5"};{\ar@{->}"v3";"v6"};
\end{xy}
\noindent Note that we use the fact that the map $\Gamma \to \mr{Sp}(4, \mb{Z}) \to \mr{Sp}(4, \mb{Z}/2\mb{Z}) \cong \Sigma_6$ sends $A, B, C, D$ to $(12), (23), (34), (45)$, which implies that the image is $\Sigma_5$. On the other hand, we have shown that $\Gamma(2) \to \Gamma$ is injective and this map factors through $\mr{Ker}(\pi_{\Gamma})$. This means that $\Gamma(2) \to \mr{Ker}(\pi_{\Gamma})$ is injective and hence is an isomorphism. Then we get a short exact sequence 
 \[
1 \to \Gamma(2) \to \Gamma \to \Sigma_5 \to 1. 
 \]
 We may take $ABCDE, BCDE, CDE, DE, E, I$ as representatives of cosets of $\Gamma$ in $\mr{Sp}(4, \mb{Z})$.
\end{proof}

\begin{rmk}
Note that there is an equality of matrices:
$
(ABCD)^5 = - I \in \Gamma(2).
$
In geometric terms, $ABCD$ is the monodromy matrix around the singularity $y^2-x^5=0$. 
\end{rmk}

\begin{rmk}
We recall that the theta group $\Gamma_{2, \theta}$ of genus $2$ is defined as
\[
\Gamma_{2, \theta} = \left\{ \left( \begin{matrix}
A & B \\
C & D
\end{matrix} \right) \in Sp(4, \mb{Z}) \Big| C D^{\mr{tr}} \ \mathrm{and} \ AB^{\mr{tr}} \ \mathrm{have \ even \ diagonal} \right\},
\]
which is a congruence subgroup of the full modular group $Sp(4, \mb{Z})$ with index $10$, see \cite[Remark 8.1, page 459]{Freitag2011}. It is also called Igusa's congruence subgroup and denoted by 
$\Gamma[1,2]$, see  \cite[page 462]{Freitag2011}. Therefore $\Gamma_{2, \theta}$ is neither the same as $\Gamma$, nor a subgroup of $\Gamma$.
\end{rmk}

\section{Differential Siegel modular forms}
\label{s5}
\subsection{The $\tmap$ map} 
\label{Boundary}
We denote the moduli space of curves of genus 2 (resp. principally polarized abelian surfaces) by 
$\ml{M}_2$ (resp. $\ml{A}_2$). Recall that we can view $\ml{M}_2$ as an open subspace of 
$\ml{A}_2$. The complement of $\ml{M}_2$ in $ \ml{A}_2$ consists of 
nodal curves, see for instance \cite[Page 1654]{2017vanGeer}.
The period map gives us an 
isomorphism $\ml{A}_2\cong {\mr{Sp}(4, \mb{Z}) \backslash \mb{H}_2}$ and under this isomorphism,
the loci of nodal curves is given by $\tau_3=0$. This in turn is the moduli of 
product of two elliptic curves.  We have the following commutative diagram:

\begin{xy}
(210,0)*+{\ml{M}_2(w)}="v1";(240,0)*+{\Gamma \backslash \mb{H}_2}="v2";
(210,-20)*+{\ml{M}_2}="v3";(240,-20)*+{\mr{Sp}(4, \mb{Z}) \backslash \mb{H}_2\cong \ml{A}_2}="v4";
{\ar@{->} "v1";"v2"};{\ar@{->} "v3";"v4"};
{\ar@{->} "v1";"v3"};{\ar@{->} "v2";"v4"}; 
\end{xy}
\noindent where the horizontal maps are injective and the vertical maps are both 6 to 1 covering maps. 
By our construction $\ml{M}_2(w)\cong \mb{P}^{2, 3, 4, 5} 
\backslash \{ t_5 \Delta = 0\}$. Smooth points of $\Delta=0$ correspond to nodal 
curves. Let us define 
\begin{eqnarray*}
\overline{\ml{M}_2(w)} &:=& \left (\mb{P}^{2, 3, 4, 5} 
\backslash \{ t_5=0\}\right) \cup  \rm smooth(\Delta = 0).
\end{eqnarray*}
We observe that the map $\overline{\ml{M}_2(w)}\to {\ml{A}_2}$ is ramified over the 
loci of nodal curves, and in this loci it is a $5$ to $1$ map.
\begin{prop}
\label{02oct2019}
 The moduli of triples $(Y,P, e)$, where $Y$ is a genus two curve, $P$ is a Weierstrass point of $Y$ and $e:=(e_1,e_2,e_3,e_4)$ is a symplectic basis of $H_1(Y,\Z)$, is isomorphic to a disjoint union of six copies of $\uhp_2-\{\tau_3=0\}$.   
\end{prop}
\begin{proof}
Let $Y_0: y^2=x^5+1$   be the hyperelliptic curve in 
\eqref{12apr2019} and $e$ be the symplectic basis of $H_1(Y_0,\Z)$ defined in 
\eqref{symbasisforfixedcurve}. This data together with the Weierstrass point $P:=[0:1:0]$ at infinity, gives us a point in the moduli space of the proposition. Let $\tilde \uhp_2$ be the connected component of this moduli containing this point.  We have the following commutative diagram 

\begin{xy}
(210,0)*+{\ml{M}_2(w)}="v1";(240,0)*+{\ml{M}_2}="v2";
(210,-20)*+{\tilde \uhp_2}="v3";(240,-20)*+{\uhp_2-\{\tau_3=0\}}="v4";
{\ar@{->} "v1";"v2"};{\ar@{->} "v3";"v4"};
{\ar@{->} "v3";"v1"};{\ar@{->} "v4";"v2"}; 
\end{xy}
\noindent
Since $\tilde\uhp_2/\Gamma\cong \ml{M}_2(w)$ and  $(\uhp_2/\Sp(4,\Z))-\{\tau_3=0\}\cong \ml{M}_2$ and $\ml{M}_2(w)\to \ml{M}_2$ is $6$ to $1$ map and $\Gamma$ has index $6$ in $\Sp(4,\Z)$ we get the fact the canonical covering $\tilde \uhp_2\to \uhp_2-\{\tau_3=0\}$ is one to one, and hence,  an isomorphism. This is enough to finish the proof.  
\end{proof}
From now on we regard $\uhp_2-\{\tau_3=0\}$ as the connected component of the  moduli  of triples $(Y,P, e)$ containing the marked point as in the beginning of the proof of Proposition \ref{02oct2019}. 
We have a natural map
$$
\tmap : \mb{H}_2-\{\tau_3=0\}  \to \Tf, 
$$
which is defined in the following way. For $\tau\in \uhp_2-\{\tau_3=0\}$ corresponding to $(Y,p,e)$,  there is a unique basis $\alpha$ of $H^1_\dR(Y)$ such that 
$(Y,\alpha)$ is enhanced and we have the following format of the period matrix
\begin{equation}
\begin{bmatrix}
\int_{e_i} \alpha_j
\end{bmatrix}_{i,j=1,2,3,4} = 
\begin{bmatrix}
\tau & -I_{2\times 2} \\
I_{2\times 2} & 0_{2\times 2} \\
\end{bmatrix}.  
\end{equation} 
For this we first take  an arbitrary enhancement $(Y, \alpha)$ and use 
the fact that there is a unique $\mbf{g} \in \BG$ such that the period matrix of $\alpha \mbf{g}$ is
of the desired form, see for instance \cite[\S 4.1]{ho18}. We define $\tmap (\tau)=(Y,\alpha)$.
We can now interpret all the functions in Proposition \ref{14.04.2019.dubai}, and in particular $T_i$'s, 
 as holomorphic  functions in $\mb{H}_2-\{\tau_3=0\}$. This is obtained by taking the pull-back 
$T_i\circ \tmap$. The $\Gm$-action on $\Sf$, which is basically given by \eqref{14april2019},  
is translated into the automorphic property of 
$T_i$'s:
\begin{equation}
\label{riverside2019}
 \det(c\tau+d)^{-\frac{i}{4}}T_{i}((a\tau+b)\cdot (c\tau+d)^{-1})=T_i(\tau),\ \ \mat{a}{b}{c}{d}\in\Gamma.
\end{equation}
The functional equation of other functions on $\Tf$ are more complicated and can be obtained from the action
of $\BG$. For computing these functional equations in the case of elliptic curves (resp. mirror quintic) see \cite[\S 8.5]{ho14} (\cite[Theorem 7]{GMCD-MQCY3}). 

In principle, the pull-back of regular 
functions in $\Tf$ by the $\tmap$-map might be meromorphic in $\tau_3=0$. Therefore, we are looking for generators of the algebra of differential Siegel modular forms for $\Gamma$:
\begin{equation}
{\rm DM}(\Gamma):= \left\{f\in\ml{O}_\Tf\ \ \Big|\ \   f\circ\tmap \hbox{ is holomorphic along  }\tau_3=0 \right\},
\end{equation}
where for an affine variety $V$ we have denoted by $\ml{O}_V$  the ring of global regular 
functions in $V$.
The algebra of Siegel modular forms for $\Gamma$ is defined similarly:
\begin{equation}
{\rm M}(\Gamma):= \left\{f\in\ml{O}_\Sf\ \ \Big|\ \   f\circ\tmap \hbox{ is holomorphic along  }\tau_3=0 \right\}.
\end{equation}
Note that using the canonical projection $\Tf\to\Sf$ we have considered $\ml{O}_{\Sf}$ as a subalgebra of $\ml{O}_\Tf$. 
We note that $T_{20},\Delta\in {\rm DM}(\Gamma)$. More precisely,  these are holomorphic Siegel modular forms for $\Gamma$ of weights $5$ and $10$, respectively. Both of them vanish along $\tau_3=0$. If not, then $\frac{1}{T_{20}}$ and $\frac{1}{\Delta}$ would be  holomorphic Siegel modular form for $\Gamma$ of negative weight which is a contradiction. 


Theoretical arguments 
as in \cite[Chapter 11]{ho2020} and \cite{2019Fonseca} show that modular vector fields in the case of abelian varieties 
are holomorphic in the whole moduli. However our vector fields in the present text have pole order one along $t_5=0$.
This translates into meromorphicity along $\tau_3=0$ and it implies that the pull-back of regular functions in $\Tf$ to
the Siegel domain are not all holomorphic along $\tau_3=0$.

\subsection{Comparison with Igusa's invariants}
In order to compare our results in this paper with Igusa's results in \cite{igusa67}, we also need to define the 
moduli space $\check \Sf$ and $\check\Tf$ which have the same definition as $\Sf$ and $\Tf$  but without the presence of Weierstrass points. 
There is a natural forgetful map:
\begin{equation} \label{compmap}
\Sf\to \check \Sf, \Tf\to \check\Tf,\ \  (Y, P, \alpha) \to (Y, \alpha).
\end{equation}
We also need the ring of invariants of binary sextics, which is given by 
\begin{equation}
I=\mb{C}[A, B, C, D, E]/\langle E^2-P(A, B, C, D)\rangle ),    
\end{equation}
where $P$ is an explicit polynomial and we have natural weights $\deg(A)=2,\ \deg B = 4, \deg(C) = 6, \deg(D) = 10,\  \deg(E) = 15$, see \cite{Bolza87}. We also need the moduli $\Tf_{\rm av}$ and $\Sf_{\rm av}$ which are the moduli of enhanced principally polarized abelian surfaces, see for instance \cite[\S 4.1]{ho18} or \cite[Chapter 11]{ho2020}.  Then we have the following diagram:
\

\begin{xy}
(180,0)*+{\mr{M}(\Gamma)}="v1";(210,0)*+{\ml{O}_{\mbf{S}}}="v2";(240,0)*+{\ml{O}_{\Tf}}="v3";(270,0)*+{{\rm DM}(\Gamma)}="v4";
(180,-20)*+{\mr{M}(\mr{Sp}(4, \mathbb{Z}))\cong  \ml{O}_{\Sf_{\rm av}}   }="v5";(210,-20)*+{\ml{O}_{\check{\mbf{S}}}}="v6";(240,-20)*+{\ml{O}_{\check{\Tf}}}="v7";(270,-20)*+{\ml{O}_{\Tf_{\rm av}}\cong{\rm DM}(\mr{Sp}(4, \mathbb{Z}))}="v8";(210,-40)*+{I}="v9";
{\ar@{^{(}->} "v1";"v2"};{\ar@{^{(}->} "v2";"v3"};{\ar@{_{(}->} "v4";"v3"};
{\ar@{^{(}->} "v5";"v6"};{\ar@{^{(}->} "v6";"v7"};{\ar@{_{(}->} "v8";"v7"};
{\ar@{_{(}->} "v5";"v1"};{\ar@{_{(}->} "v6";"v2"};{\ar@{_{(}->} "v7";"v3"};{\ar@{_{(}->} "v8";"v4"};
{\ar@{_{(}->} "v9";"v6"};
\end{xy}
In \cite[Page 848]{igusa67}, Igusa shows that  the intersection of of ${\rm M}(\mr{Sp}(4, \mb{Z}))$ 
with  $I$, is generated by \[
E_4:=B,\  E_6:=4AB-3C, \ \chi_{10}:=D,\  \chi_{12}:=AD,\ \chi_{35}:=D^2E 
\]
up to constants.
Note that in the proof of Theorem \ref{moduli sp for Siegel}, the $\mb{G}_m$-action multiplies $[\frac{\mr{d}x}{y}] \wedge [\frac{x\mr{d}x}{y}]$ with $a^{-4}$ for $a \in \mb{G}_m$. This implies that the degree in our case is $4$ times the canonical  degree of Siegel modular forms. 
Let $F$ be the composition of two inclusions $I\hookrightarrow \ml{O}_{\check{\Sf}}\hookrightarrow \ml{O}_{\Sf}$.
Looking $I$ inside $\ml{O}_{\Sf}$, we know that $E_4$ (resp. $E_6, \chi_{10}, \chi_{12}, \chi_{35}$) are degree $16$ (resp. $24, 40, 48, 140$) polynomials in variables $T_4, T_8, T_{12}, T_{16}$ and $T_{20}$, which are defined in Remark \ref{T-variables}. 
We can use invariant theory to determine these expressions explicitly. This is as follows. The curve given by $y^2 = x^5+t_2x^3+t_3x^2+t_4x+t_5$ corresponds to the binary sextic $f = t_5x^6+t_4x^5y+t_3x^4y^2+t_2x^3y^3+xy^5$.
The result up to multiplication with constants is
\begin{eqnarray*}
A(f)  &=& -3t_2^2-20t_4, \\
B(f) &=& -3t_2t_3^2+9t_2^2t_4-20t^2_4+75t_3t_5,\\
C(f) &=& 12t_2^3t_3^2+18t_3^4-36t_2^4t_4-13t_2t_3^2t_4-88t_2^2t_4^2+\\
     & &160t_4^3-165t_2^2t_3t_5-800t_3t_4t_5-1125t_2t_5^2, \\
D(f) &=& \Delta \ \  \hbox{ is defined in \eqref{23.02.2019}, }\\
E(f) &=& 
729t_{2}^{10}t_{5}^{2}-486t_{2}^{9}t_{3}t_{4}t_{5}+108t_{2}^{8}t_{3}^{3}t_{5}+81t_{2}^{8}t_{3}^{2}t_{4}^{2}-18225t_{2}^{8}t_{4}t_{5}^{2}-36t_{2}^{7}t_{3}^{4}t_{4}+\\
& &12150t_{2}^{7}t_{3}^{2}t_{5}^{2}+10800t_{2}^{7}t_{3}t_{4}^{2}t_{5}+4t_{2}^{6}t_{3}^{6}-9720t_{2}^{6}t_{3}^{3}t_{4}t_{5}-1800t_{2}^{6}t_{3}^{2}t_{4}^{3}+\\
&&135000t_{2}^{6}t_{4}^{2}t_{5}^{2}+1584t_{2}^{5}t_{3}^{5}t_{5}+2015t_{2}^{5}t_{3}^{4}t_{4}^{2}-175500t_{2}^{5}t_{3}^{2}t_{4}t_{5}^{2}-60000t_{2}^{5}t_{3}t_{4}^{3}t_{5}+\\
&&928125t_{2}^{5}t_{5}^{4}-623t_{2}^{4}t_{3}^{6}t_{4}+60000t_{2}^{4}t_{3}^{4}t_{5}^{2}+92500t_{2}^{4}t_{3}^{3}t_{4}^{2}t_{5}+10000t_{2}^{4}t_{3}^{2}t_{4}^{4}-\\
&&1012500t_{2}^{4}t_{3}t_{4}t_{5}^{3}-250000t_{2}^{4}t_{4}^{3}t_{5}^{2}+59t_{2}^{3}t_{3}^{8}-45050t_{2}^{3}t_{3}^{5}t_{4}t_{5}-17500t_{2}^{3}t_{3}^{4}t_{4}^{3}+\\
&&225000t_{2}^{3}t_{3}^{3}t_{5}^{3}+850000t_{2}^{3}t_{3}^{2}t_{4}^{2}t_{5}^{2}-2812500t_{2}^{3}t_{4}t_{5}^{4}+5700t_{2}^{2}t_{3}^{7}t_{5}+10825t_{2}^{2}t_{3}^{6}t_{4}^{2}-\\
&&478125t_{2}^{2}t_{3}^{4}t_{4}t_{5}^{2}-50000t_{2}^{2}t_{3}^{3}t_{4}^{3}t_{5}+1875000t_{2}^{2}t_{3}^{2}t_{5}^{4}+1250000t_{2}^{2}t_{3}t_{4}^{2}t_{5}^{3}-\\
&&2610t_{2}t_{3}^{8}t_{4}+93750t_{2}t_{3}^{6}t_{5}^{2}+32500t_{2}t_{3}^{5}t_{4}^{2}t_{5}-1187500t_{2}t_{3}^{3}t_{4}t_{5}^{3}+216t_{3}^{10}-\\
&&9000t_{3}^{7}t_{4}t_{5}+625t_{3}^{6}t_{4}^{3}+175000t_{3}^{5}t_{5}^{3}+15625t_{3}^{4}t_{4}^{2}t_{5}^{2}-390625t_{3}^{2}t_{4}t_{5}^{4}-9765625t_{5}^{6},
\end{eqnarray*}
which can be easily rewritten in terms of $T_i$'s. 
The equalities for $A,B,C,D$ are taken from \cite[\S 3.3 page 47]{2013Taylor}.
Igusa's generator $E$ is the first invariant with the odd degree
and we have computed it from the classical theory of transvectants. Up to constant, it is $j_{15}$ defined in  \cite[page 53]{2014Draissma}.

For our main purposes, it would be essential to generalize Igusa's results and find generators for the ${\rm DM}(\Gamma)$. For this one has to understand in a geometric way why the pull-back of $A=\frac{AD}{D}$ and $E=\frac{D^2E}{D^2}$ by the $\tmap$-map has a pole order $2$ and $4$ along $\tau_3=0$, respectively. Note that $D$ has a zero of order two along $\tau_3=0$, see for instance \cite[Page 849]{igusa67}. 
One possible way to investigate this is as follows. 
We take a three parameters family of genus two curves with a marked Weierstrass point (for instance the family $Y_t$ in \eqref{quintic} with some $t_i$ equal to zero). We also take the canonical basis $\omega$ of $H^1_\dR(Y_t)$ in \eqref{de Rham basis1}. We can  express the generators of $\ml{O}_\Tf$ in terms of the periods of $\omega$ and then we investigate when such period expressions are holomorphic around nodal curves. This is going to be explained in \S\ref{7oct2019}. Similar computations has been explained in the case of elliptic curve in \cite[\S 8.7]{ho14} and mirror quintic in \cite[\S 4.9]{GMCD-MQCY3}.

\subsection{Expressions of the image of $\tmap$-map via abelian integrals}
\label{7oct2019}
We choose the following loci $L$ of $\Tf$  consisting of points $(t,S_0)$ of the form:
\begin{equation}
\label{lunchwithyau2019}
(t_2,t_3,t_4,t_5),\ S_0=
\left( 
\begin{matrix}
1 & 0 & 0 & 0 \\
0 & 1 & 0 & 0 \\
0 & \frac{t_2}{4} & 0 & \frac{3}{4} \\
0 & 0 &  \frac{1}{4} & 0 
\end{matrix}
\right).
\end{equation}
We could restrict ourselves to a three dimensional subspace in $t=(t_2,t_3,t_4,t_5)$, however, for lack of a canonical choice we do not do it. 
Let $$x_{ij}:=\int_{e_i}{\omega_j},$$
where $e_i$ (resp. $\omega_j$) are the symplectic basis 
(resp. the canonical de Rham cohomology basis) of genus two curves determined by (\ref{quintic}). 
We have 
$$
\left[
\int_{e_i}\alpha_j
\right]=[x_{ij}]S_0^{\tr}= 
\begin{bmatrix}
\tau & -I \\
I & 0
\end{bmatrix}\cdot \mbf{g},
$$
where 
\begin{equation}
\label{deadsleap2019}
\tau=\begin{bmatrix}
      \frac{x_{11}x_{42}-x_{41}x_{12}}{x_{31}x_{42}-x_{41}x_{32}} & -\frac{x_{11}x_{32}-x_{31}x_{12}}{x_{31}x_{42}-x_{41}x_{32}}  \\
      \frac{x_{21}x_{42}-x_{41}x_{22}}{x_{31}x_{42}-x_{41}x_{32}} & -\frac{x_{21}x_{32}-x_{31}x_{22}}{x_{31}x_{42}-x_{41}x_{32}} 
      \end{bmatrix},\ \ 
\mbf{g} =\begin{bmatrix}
x_{31} & x_{32} & \frac{1}{4}t_2x_{32}+\frac{3}{4} x_{34} & \frac{1}{4}x_{33} \\
x_{41} & x_{42} & \frac{1}{4}t_2x_{42}+\frac{3}{4} x_{44} & \frac{1}{4}x_{43} \\
0 & 0 & \frac{x_{42}}{x_{31}x_{42} - x_{32}x_{41}} & \frac{-x_{41}}{x_{31}x_{42} - x_{32}x_{41}}  \\
0 & 0 & \frac{-x_{32}}{x_{31}x_{42} - x_{32}x_{41}} & \frac{x_{31}}{x_{31}x_{42} - x_{32}x_{41}}  \\
\end{bmatrix}\in\BG.
\end{equation}
This implies that $\tmap(\tau) =(t,S_0) \bullet \mbf{g}^{-1}$, where $(t,S_0)$ is the point in $\Tf$ given by (\ref{lunchwithyau2019}).
This means that the image of $\tau$ as above (viewed as an element in the Siegel domain $\mb{H}_2$) under the $\tmap$-map is given by the entries of the action of $\mbf{g}^{-1}$ on the point $(t,S_0)$ in \eqref{lunchwithyau2019}, where $t=(t_2,t_3,t_4,t_5)$.  This is  $(t,S)$, where 
$$
S:=\mbf{g}^{-\mr{tr}}S_0=
$$
\begin{equation}
\left[
\begin{array}{*{4}{c}}
\frac{x_{42}}{x_{31}x_{42}-x_{32}x_{41}} & -\frac{x_{41}}{x_{31}x_{42}-x_{32}x_{41}} & 0 & 0 \\
-\frac{x_{32}}{x_{31}x_{42}-x_{32}x_{41}} & \frac{x_{31}}{x_{31}x_{42}-x_{32}x_{41}} & 0 & 0 \\
\frac{3x_{31}x_{32}x_{44}-3x_{31}x_{34}x_{42}+x_{32}^{2}x_{43}-x_{32}x_{33}x_{42}}{4x_{31}x_{42}-4x_{32}x_{41}} & -\frac{3x_{31}^{2}x_{44}+x_{31}x_{32}x_{43}-3x_{31}x_{34}x_{41}-x_{32}x_{33}x_{41}}{4x_{31}x_{42}-4x_{32}x_{41}} & \frac{x_{32}}{4} & \frac{3x_{31}}{4} \\
\frac{3x_{32}x_{41}x_{44}+x_{32}x_{42}x_{43}-x_{33}x_{42}^{2}-3x_{34}x_{41}x_{42}}{4x_{31}x_{42}-4x_{32}x_{41}} & -\frac{3x_{31}x_{41}x_{44}+x_{31}x_{42}x_{43}-x_{33}x_{41}x_{42}-3x_{34}x_{41}^{2}}{4x_{31}x_{42}-4x_{32}x_{41}} & \frac{x_{42}}{4} & \frac{3x_{41}}{4}
\end{array}
\right]
\end{equation}

Now, a differential Siegel modular form $f\in\ml{O}_\Tf$ evaluated at $(t,S)$ above has an expression in terms 
of abelian integrals $x_{ij}$.  If we regard $f(\tau)$ as a function in $\tau$, this means that we replace $\tau$ in \eqref{deadsleap2019} in $f(\tau)$ and we get such period expressions.  For instance, we have the equality:
\begin{equation}
 T_{4i}
 \left(
 \begin{bmatrix}
      \frac{x_{11}x_{42}-x_{41}x_{12}}{x_{31}x_{42}-x_{41}x_{32}} & -\frac{x_{11}x_{32}-x_{31}x_{12}}{x_{31}x_{42}-x_{41}x_{32}}  \\
      \frac{x_{21}x_{42}-x_{41}x_{22}}{x_{31}x_{42}-x_{41}x_{32}} & -\frac{x_{21}x_{32}-x_{31}x_{22}}{x_{31}x_{42}-x_{41}x_{32}} 
      \end{bmatrix}
 \right)=
 a_i \cdot (x_{31}x_{42}-x_{41}x_{32})^{i},\ \ i=1,2,3,4,5,
\end{equation}
where $(a_1,a_2,a_3,a_4,a_5) = (t_2,t_4,t_3^2,t_3t_5,t_5^2)$. This follows form the definition of $T_{i}$'s in \S\ref{affinechart} and we have regarded them  as functions on $\uhp_2$. The periods $x_{ij}$ are functions in $t$, and if we write them in terms of the modular parameters \eqref{6oct2016} and then these parameters as functions of $\tau$ then we may arrive in formulas which are natural generalization of some classical identities such as Fricke and Klein's formula in the case of elliptic curves:
$$
\sqrt[4]{E_4(\tau)}=F(\frac{1}{12},\frac{5}{12},1;\frac{1728}{j(\tau)}),
$$
see for instance \cite[\S 2.3]{kon}. See also \cite[page 364]{ho14} for more examples of this type.  
Finally note that we have 
six relations between $x_{ij}$:
\begin{eqnarray}
\label{polyrelations}
x_{12}x_{31}-x_{11}x_{32}+x_{22}x_{41}-x_{21}x_{42}  &= &0,\\ \nonumber
x_{13}x_{31}-x_{11}x_{33}+x_{23}x_{41}-x_{21}x_{43} & =& 0,\\ \nonumber
x_{14}x_{31}-x_{11}x_{34}+x_{24}x_{41}-x_{21}x_{44} &=& \frac{4}{3},\\ \nonumber
x_{13}x_{32}-x_{12}x_{33}+x_{23}x_{42}-x_{22}x_{43}&=& 4,\\ \nonumber
x_{14}x_{32}-x_{12}x_{34}+x_{24}x_{42}-x_{22}x_{44}  &=& 0,\\  \nonumber
x_{14}x_{33}-x_{13}x_{34}+x_{24}x_{43}-x_{23}x_{44}  &=& \frac{4}{3}t_2.
\end{eqnarray}
These equalities correspond to the entries $(1,2),(1,3),(1,4),(2,3),(2,4)$ and  $(3,4)$ of
the equality $[x_{ij}]^{\tr}\Phi^{-\tr} [x_{ij}]=\Omega$, where $\Phi$ and $\Omega$ are given in \eqref{05oct2019}  and \eqref{Phi-omega}. This comes from the duality of intersection form in homology and cup product in cohomology, see for instance \cite[\S 4.1]{GMCD-MQCY3}.

\subsection{Proof of Theorem \ref{maintheo}}
The main ingredient of the proof is the $\tmap$-map which translates most of the algebraic machinery introduced in this paper into complex analysis of holomorphic functions on the Siegel domain. The $153$ quantities mentioned in the theorem are just the pull-back of functions constructed in \S\ref{affinechart}. 
Part \ref{mt1} is established in \S\ref{Boundary}. 
Note that the $\Gm$-action on $\Sf$ and the functional equation  \eqref{14april2019} via the $\tmap$-map is translated into the functional equation \eqref{riverside2019}.
Part \ref{mt2} follows from the fact that $\Tf$ is an affine variety defined over $\Q$. The ideal of polynomial relations among $X_i$ can be computed easily from their algebraic definition in \S \ref{affinechart}.
For part \ref{mt3} note that the affine variety $\Q[X]/I$ is just the ring of functions that we have described in \S\ref{affinechart}. For part \ref{mt4} note that 
$\frac{\partial X_i}{\partial \tau_i}$ is just $dX_i(\Ra_k)$ and the computation of this is explained in \S\ref{JH2019}. In \S\ref{JH2019} we have explained how to write $\Ra_k$'s in the affine chart given by $X_i$'s. Since $\Ra_k$ is meromorphic in $t_5$, the vector field in $X_i$ variables has a pole order one in $X_5=T_{20}=t_5^2$.

\section{Compatibility with Resnikoff's computation}
In this section, we compare our construction of differential equations of Siegel modular forms  with few differential equations obtained by  Resnikoff in \cite{Resnikoff1970-1, Resnikoff1970-2}. 
Let
$$
\partial:=\det 
\begin{bmatrix}
\Ra_1 & \frac{1}{2} \Ra_3 \\
\frac{1}{2} \Ra_3 & \Ra_2
\end{bmatrix}=\Ra_1\circ\Ra_2-\frac{1}{4}\Ra_3\circ \Ra_3,
$$
where we consider $\Ra_i:{\mathcal O}_\Tf\to{\mathcal O}_\Tf$ as derivations. 
For $f\in{\mathcal O}_\Sf$ of degree $4w$ (of weight $w$ if $f$ is interpreted as Siegel modular forms) let also 
\begin{equation}
\begin{split}
D^{\rm Res}f = \frac{4w-1}{4w^2} f\partial f+ \frac{1-2w}{8w^2} \partial f^2.
\end{split}
\end{equation}
It turns out that $D^{\rm Res}$ maps ${\mathcal O}_\Sf$ to the set of cuspidal forms in ${\mathcal O}_\Sf$. It seems to be interesting to prove this statement in our geometric framework.
For 
$$
\widehat{E}_2 = A/\delta^2, \widehat{E}_4 = B/\delta^4, \widehat{E}_6 =(4AB-3C)/\delta^6, \widehat{\chi}_{10}  = D/\delta^{10}, \widehat{\chi}_{12} = AD/\delta^{12}.
$$ 
we have verified the following equalities which confirms the above statement:
\begin{eqnarray}
D^{\rm Res} \widehat{E}_4 &=& \frac{984375}{1024} \widehat{\chi}_{10}; \\
D^{\rm Res} \widehat{E}_6 &=& \frac{2165625}{64}\widehat{E}_4\widehat{\chi}_{10}; \\
D^{\rm Res} \widehat{\chi}_{10} &=& -\frac{3}{6400}\widehat{\chi}_{10}\widehat{\chi}_{12}; \\
D^{\rm Res} \widehat{\chi}_{12} &=& -\frac{49}{2304}\widehat{E}_6\widehat{\chi}_{10}^2+\frac{37}{2304}\widehat{E}_4\widehat{\chi}_{10}\widehat{\chi}_{12}.\\
D^{\rm Res} \widehat{E}_2  &=& \frac{3}{256} \widehat{E}_2^3 -\frac{1}{16}\widehat{E}_2\widehat{E}_4 - 
\frac{1}{8}\widehat{E}_6.
\end{eqnarray}
Except for the last one,  these are also obtained in  \cite[Theorem 1]{Resnikoff1970-1}).
Note that up to multiplication with a constant $\widehat{E}_4$ (resp. $\widehat{E}_6, \widehat{\chi}_{10}, \widehat{\chi}_{12}$) is equal to $E_4$ (resp. $E_6, \chi_{10}, \chi_{12}$) and this is the reason why the constants in the above equalities are different from Resnikoff's constants.

\section{Final comments}
The vector fields $R_k$'s together with the vector fields coming from the action of $\BG$ on $\Tf$
provide natural foliations for  Humbert surfaces in the moduli of principally polarized abelian varieties.
The general theory is being formulated in \cite{ho2020} and explicit computations of $R_k$'s in the present text 
can be used for computing explicit equations for such surfaces, see \cite{GMCD-K3}. 
In \cite{igusa67}, Igusa defined the map from the ring of Siegel modular forms to the ring of invariants via the theta functions. It seems also reasonable to use theta functions to understand the differential Siegel modular forms in our sense. For this we must rewrite the content of this paper for the family $y^2=\prod_{i=1}^6(x-t_i)$.
In our geometric setting, we have not written  (differential) Siegel modular forms as Poincar\'e series and  
one might ask for their  Fourier expansions. This can be done using the explicit 
expression of $R_k$'s, which turns out to be recursion in the coefficients of such Fourier expansions. For 
such a computation in the case of Calabi-Yau modular forms, one may find in \cite[Chapter 5]{GMCD-MQCY3}.

\appendix 
\section{Our computer code}
\label{03o2019}
For the convenience of the reader we reproduce here our
computer code in order to compute the Gauss-Manin connection matrices $B_i, \ i=2,3,4,5$ and the vector fields $R_k, \ k=1,2,3$. The library {\tt foliation.lib} of {\sc Singular} can be downloaded from the second author's webpage.  
\begin{tiny}
\begin{verbatim}
LIB "foliation.lib"; LIB "latex.lib";
ring r=(0,t(2..5)),(x,y),wp(2,5);
poly f=y^2-x^5-t(2)*x^3-t(3)*x^2-t(4)*x-t(5);
list GMl=gaussmaninmatrix(f,list(t(2..5)), 2);
matrix S[4][4]=0,0,0,1,0,0,1,0,0,1,0,0,1,0,0,0; print(S);
     for (int i=2; i<=5;i=i+1){ GMl[i]=S*GMl[i]*inverse(S);}
list B=0; for (int i=1; i<=4;i=i+1){ B=insert(B, GMl[1]*GMl[i+1], size(B));}
poly Delta=1/GMl[1];
//-----B is the list of Gauss-Manin connection matrices-----------
matrix Om[4][4]= 0,0,0,4/3,0,0,4,0,0,-4,0,(4/3)*t(2),-4/3,0, 
-(4/3)*t(2), 0;   print(Om);
//-----Om is the cup product pairing in de Rham cohomology
//-----Checking the differential equation of Om--------------
for (int i=2; i<=5;i=i+1){ print(B[i]*Om+Om*transpose(B[i]));}
//-------A check for the structure of M matrix----------
matrix M[4][4]=B[2][1,3],B[3][1,3], B[4][1,3], B[5][1,3], 
B[2][1,4],B[3][1,4], B[4][1,4], B[5][1,4], B[2][2,3],
B[3][2,3], B[4][2,3], B[5][2,3], B[2][2,4], B[3][2,4], B[4][2,4], B[5][2,4];
matrix kn[1][4]=-3,0,0,1; print(kn*M);
//-------Computing R_k's in t_i's---------------------------
ring r2=(0,t(2..5), s(-2), s(-4),s(-6),s(1..4)(1..2)),(x(2..5)),dp;
matrix M=imap(r, M);
matrix X[4][1]=x(2..5);  matrix S[4][1]=s(-4),s(-6),s(-2), 3*s(-4);
ideal I=M*X-S; I=std(I); int i; list Ra=0;
for (i=2; i<=5;i=i+1){ Ra=insert(Ra, reduce(x(i), I),size(Ra));}
for (i=2; i<=5;i=i+1){Ra[i]=subst(Ra[i],x(5),0); }
list B=imap(r, B);
matrix E=Ra[2]*B[2]+Ra[3]*B[3]+Ra[4]*B[4]+Ra[5]*B[5]; //-Computing E-2, E-4, E-6--
matrix S1[2][2]=s(1)(1),s(1)(2),s(2)(1),s(2)(2);
matrix S3[2][2]=s(3)(1),s(3)(2),s(4)(1),s(4)(2);
poly delta= det(S1);
matrix Om2[2][2]= 0,4/3,4,0;
matrix S4 =inverse(transpose(S1))*inverse(transpose(Om2));
//---------Computing R_1, R_2, R_3--------------------------
list Cl; matrix C[2][2]=1,0,0,0; Cl=C;
C=0,0,0,1; Cl=insert(Cl, C, size(Cl)); 
C=0,1,1,0; Cl=insert(Cl, C, size(Cl));
list s4l=-1/4*s(2)(1)*s(2)(2)/(delta^2), -1/4*s(1)(1)*s(1)(2)/(delta^2),  
                                   1/4*(s(1)(2)*s(2)(1)+s(1)(1)*s(2)(2))/(delta^2);
list s6l=3/4*s(2)(2)^2/(delta^2),        3/4*s(1)(2)^2/(delta^2),         -3/2*s(1)(2)*s(2)(2)/(delta^2);
list s2l=1/4*s(2)(1)^2/(delta^2),        1/4*s(1)(1)^2/(delta^2),         -1/2*s(1)(1)*s(2)(1)/(delta^2);
matrix Eh=E; int k;  matrix E1[2][2]; matrix E3[2][2]; matrix R[4][2]; list Ral; list Rh; poly komak; 
for (k=1; k<=3;k=k+1)
     { 
     Rh=Ra[2],Ra[3],Ra[4],Ra[5];
     for (i=1; i<=size(Rh);i=i+1)
         {
         komak=Rh[i]; 
         komak=substpar(komak, s(-2), s2l[k]); komak=substpar(komak, s(-4), s4l[k]); 
         komak=substpar(komak, s(-6), s6l[k]);
         Rh[i]=komak; 
         }
     Eh=substpar(E, s(-2), s2l[k]); Eh=substpar(Eh, s(-4), s4l[k]); Eh=substpar(Eh, s(-6), s6l[k]);
     E1=submat(Eh, 1..2, 1..2); E3=submat(Eh, 3..4, 1..2); 
     R=Cl[k]*S3-S1*E1, -S3*E1-S4*E3;
     Rh=Rh+list(R[1,1], R[1,2], R[2,1], R[2,2], R[3,1], R[3,2], R[4,1], R[4,2]); 
     Ral=insert(Ral, Rh, size(Ral)); 
     }
//--------------Tangency to Delta=0 for this we can use s(-4),s(-6), s(-2)------------
poly Delta=imap(r, Delta);
(diffpar(Delta,t(2))*Ra[2]+diffpar(Delta,t(3))*Ra[3]+diffpar(Delta,t(4))*Ra[4]+diffpar(Delta,t(5))*Ra[5])/Delta;     
//-------------------Three proceedure----------------------------
proc dvpar(poly P, list parl, list vfl)
{
poly Q; for (int i=1; i<=size(parl);i=i+1){Q=Q+diffpar(P, parl[i])*vfl[i];} return(Q); 
}
proc Resnikoffd(poly P)
{
poly Q= dvpar( dvpar(P, lp, Ral[1]), lp, Ral[2])-(1/4)*dvpar( dvpar(P, lp, Ral[3]), lp, Ral[3]); return(Q); 
} 
proc ResnikoffD(poly P, int w)
{
poly Q= (8*w-2)*P*Resnikoffd(P)-(2*w-1)*Resnikoffd(P^2); Q=Q/(8*w^2); return(Q); 
}
//------verfying that F=0 is invariant-----------------
poly F=s(4)(2)*s(2)(1)-s(4)(1)*s(2)(2)+s(3)(2)*s(1)(1)-s(3)(1)*s(1)(2)-t(2)/4; 
list lp=t(2..5), s(1)(1), s(1)(2), s(2)(1), s(2)(2), s(3)(1), s(3)(2), s(4)(1), s(4)(2);
for (k=1; k<=3;k=k+1)
        {
        delta^2*25*t(5)*dvpar(F, lp, Ral[k])/F;     
        }
//----------------Verifying Resnikoff's differential equation-----------
poly Ai=(-3*t(2)^2-20*t(4))/delta^2;
poly Bi= (-3*t(2)*t(3)^2+9*t(2)^2*t(4)-20*t(4)^2+75*t(3)*t(5))/(delta^4); 
poly Ci=(12*t(2)^3*t(3)^2+18*t(3)^4-36*t(2)^4*t(4)-13*t(2)*t(3)^2*t(4)-88*t(2)^2*t(4)^2+160*t(4)^3-
165*t(2)^2*t(3)*t(5)-800*t(3)*t(4)*t(5)-1125*t(2)*t(5)^2)/(delta^6);
poly Di=Delta/delta^10;
poly Ei=
(729*t(2)^10*t(5)^2-486*t(2)^9*t(3)*t(4)*t(5)+108*t(2)^8*t(3)^3*t(5)+81*t(2)^8*t(3)^2*t(4)^2
-18225*t(2)^8*t(4)*t(5)^2-36*t(2)^7*t(3)^4*t(4)+12150*t(2)^7*t(3)^2*t(5)^2+10800*t(2)^7*t(3)*t(4)^2*t(5)
+4*t(2)^6*t(3)^6-9720*t(2)^6*t(3)^3*t(4)*t(5)-1800*t(2)^6*t(3)^2*t(4)^3+135000*t(2)^6*t(4)^2*t(5)^2
+1584*t(2)^5*t(3)^5*t(5)+2015*t(2)^5*t(3)^4*t(4)^2-175500*t(2)^5*t(3)^2*t(4)*t(5)^2-60000*t(2)^5*t(3)*t(4)^3*t(5)
+928125*t(2)^5*t(5)^4-623*t(2)^4*t(3)^6*t(4)+60000*t(2)^4*t(3)^4*t(5)^2+92500*t(2)^4*t(3)^3*t(4)^2*t(5)
+10000*t(2)^4*t(3)^2*t(4)^4-1012500*t(2)^4*t(3)*t(4)*t(5)^3-250000*t(2)^4*t(4)^3*t(5)^2+59*t(2)^3*t(3)^8
-45050*t(2)^3*t(3)^5*t(4)*t(5)-17500*t(2)^3*t(3)^4*t(4)^3+225000*t(2)^3*t(3)^3*t(5)^3
+850000*t(2)^3*t(3)^2*t(4)^2*t(5)^2-2812500*t(2)^3*t(4)*t(5)^4+5700*t(2)^2*t(3)^7*t(5)+10825*t(2)^2*t(3)^6*t(4)^2
-478125*t(2)^2*t(3)^4*t(4)*t(5)^2-50000*t(2)^2*t(3)^3*t(4)^3*t(5)+1875000*t(2)^2*t(3)^2*t(5)^4
+1250000*t(2)^2*t(3)*t(4)^2*t(5)^3-2610*t(2)*t(3)^8*t(4)+93750*t(2)*t(3)^6*t(5)^2+32500*t(2)*t(3)^5*t(4)^2*t(5)
-1187500*t(2)*t(3)^3*t(4)*t(5)^3+216*t(3)^10-9000*t(3)^7*t(4)*t(5)+625*t(3)^6*t(4)^3+175000*t(3)^5*t(5)^3
+15625*t(3)^4*t(4)^2*t(5)^2-390625*t(3)^2*t(4)*t(5)^4-9765625*t(5)^6)/delta^15;

poly DBi=ResnikoffD(Bi,4);
substpar(DBi/Di,t(2), 4*F+t(2));
poly DAD=ResnikoffD(Di,10);
substpar(DAD/(Di*Ai*Di), t(2), 4*F+t(2));
poly P=ResnikoffD(Ai*Bi+(-3/4)*Ci,6);
poly X1=substpar(P/(Bi*Di),t(2), 4*F+t(2));

poly DK12=ResnikoffD(Ai*Di,12);
poly X1=substpar(DK12/(Di)^2,t(2), 4*F+t(2));
X1=X1*delta^6;
ring r3=(0,t(2..5)), (s(1..4)(1..2)), dp;
ideal I=(-t(2)+4*s(1)(1)*s(3)(2)-4*s(1)(2)*s(3)(1)+4*s(2)(1)*s(4)(2)-4*s(2)(2)*s(4)(1))/4;
I=std(I);  poly X1=imap(r2, X1); reduce(X1, I); 
\end{verbatim}
\end{tiny}

\section{Three modular vector fields}
In this appendix, we reproduce our computation of modular vector fields $\Ra_i,\ i=1,2,3$ corresponding to classical derivations $\frac{\partial}{\partial \tau_i}$ in the Siegel domain for $g=2$. 
They are written in the weighted  homogeneous coordinate system 
$t_2,t_3,t_4,t_5$,  $s_{11}, s_{12}, s_{21},$ $s_{22}, s_{31}, s_{32}, s_{41}, s_{42}$
with $\delta:=s_{11}s_{22}-s_{12}s_{21}$. 
\begin{eqnarray*}
&& \Ra_1= \\
&& [(4t_2^2t_4s_{21}^2+30t_2^2t_5s_{21}s_{22}-6t_2t_3t_4s_{21}s_{22}-20t_2t_3t_5s_{22}^2+8t_2t_4^2s_{22}^2+75t_3t_5s_{21}^2 \\
&& -100t_4t_5s_{21}s_{22}+125t_5^2s_{22}^2)/(50t_5\delta^2)]\partial/\partial t_2 \\
&&+ [(-30t_2^2t_5s_{21}^2+6t_2t_3t_4s_{21}^2+65t_2t_3t_5s_{21}s_{22}-10t_2t_4t_5s_{22}^2-9t_3^2t_4s_{21}s_{22}\\
&&-30t_3^2t_5s_{22}^2+12t_3t_4^2s_{22}^2+100t_4t_5s_{21}^2-125t_5^2s_{21}s_{22})/(50t_5\delta^2)]\partial/\partial t_3 \\
&&+ [(-20t_2t_3t_5s_{21}^2+8t_2t_4^2s_{21}^2+10t_2t_4t_5s_{21}s_{22}+75t_2t_5^2s_{22}^2+30t_3^2t_5s_{21}s_{22} \\
&&-12t_3t_4^2s_{21}s_{22}-55t_3t_4t_5s_{22}^2+16t_4^3s_{22}^2+125t_5^2s_{21}^2)/(50t_5\delta^2)]\partial/\partial t_4 \\
&&+[0] \partial/\partial t_5 \\
&&+[(6t_2t_4s_{11}s_{21}^2+20t_2t_5s_{11}s_{21}s_{22}-20t_2t_5s_{12}s_{21}^2-9t_3t_4s_{11}s_{21}s_{22}-30t_3t_5s_{11}s_{22}^2\\
&&
+30t_3t_5s_{12}s_{21}s_{22}+12t_4^2s_{11}s_{22}^2-15t_4t_5s_{12}s_{22}^2 +100t_5s_{11}^2s_{22}^2s_{31} \\
&&-200t_5s_{11}s_{12}s_{21}s_{22}s_{31}+100t_5s_{12}^2s_{21}^2s_{31})/(100t_5\delta^2)]\partial/\partial s_{11} \\
&&+ [(2t_2t_4s_{12}s_{21}^2-25t_2t_5s_{11}s_{22}^2+40t_2t_5s_{12}s_{21}s_{22}-3t_3t_4s_{12}s_{21}s_{22}-10t_3t_5s_{12}s_{22}^2\\
&&+4t_4^2s_{12}s_{22}^2+100t_5s_{11}^2s_{22}^2s_{32}-200t_5s_{11}s_{12}s_{21}s_{22}s_{32}+25t_5s_{11}s_{21}^2\\
&&+100t_5s_{12}^2s_{21}^2s_{32})/(100t_5\delta^2)]\partial/\partial s_{12} \\
&&+[(6t_2t_4s_{21}^3-9t_3t_4s_{21}^2s_{22}+12t_4^2s_{21}s_{22}^2-15t_4t_5s_{22}^3)/(100t_5\delta^2)]\partial/\partial s_{21}\\
&&+[(2t_2t_4s_{21}^2s_{22}+15t_2t_5s_{21}s_{22}^2-3t_3t_4s_{21}s_{22}^2-10t_3t_5s_{22}^3+4t_4^2s_{22}^3\\
&&+25t_5s_{21}^3)/(100t_5\delta^2)]\partial/\partial s_{22}\\
&&+[(6t_2t_4s_{11}s_{21}^2s_{22}s_{31}-6t_2t_4s_{12}s_{21}^3s_{31}-20t_2t_5s_{11}s_{21}^2s_{22}s_{32} +20t_2t_5s_{11}s_{21}s_{22}^2s_{31}\\ 
&&+20t_2t_5s_{12}s_{21}^3s_{32}-20t_2t_5s_{12}s_{21}^2s_{22}s_{31}-9t_3t_4s_{11}s_{21}s_{22}^2s_{31}+9t_3t_4s_{12}s_{21}^2s_{22}s_{31} \\
&&+30t_3t_5s_{11}s_{21}s_{22}^2s_{32}-30t_3t_5s_{11}s_{22}^3s_{31}-30t_3t_5s_{12}s_{21}^2s_{22}s_{32}+30t_3t_5s_{12}s_{21}s_{22}^2s_{31} \\
&&+12t_4^2s_{11}s_{22}^3s_{31}-12t_4^2s_{12}s_{21}s_{22}^2s_{31}-15t_4t_5s_{11}s_{22}^3s_{32}+15t_4t_5s_{12}s_{21}s_{22}^2s_{32} \\
&&+25t_4t_5s_{21}^2s_{22}-50t_5^2s_{21}s_{22}^2)/(100t_5\delta^3)] \partial/\partial s_{31} \\
&&+[(2t_2t_4s_{11}s_{21}^2s_{22}s_{32}-2t_2t_4s_{12}s_{21}^3s_{32}+40t_2t_5s_{11}s_{21}s_{22}^2s_{32}-25t_2t_5s_{11}s_{22}^3s_{31} \\
&&-40t_2t_5s_{12}s_{21}^2s_{22}s_{32}+25t_2t_5s_{12}s_{21}s_{22}^2s_{31}+10t_2t_5s_{21}^3-3t_3t_4s_{11}s_{21}s_{22}^2s_{32} \\
&&+3t_3t_4s_{12}s_{21}^2s_{22}s_{32}-10t_3t_5s_{11}s_{22}^3s_{32}+10t_3t_5s_{12}s_{21}s_{22}^2s_{32}+35t_3t_5s_{21}^2s_{22} \\
&&+4t_4^2s_{11}s_{22}^3s_{32}-4t_4^2s_{12}s_{21}s_{22}^2s_{32}-55t_4t_5s_{21}s_{22}^2+75t_5^2s_{22}^3+25t_5s_{11}s_{21}^2s_{22}s_{31}\\
&&-25t_5s_{12}s_{21}^3s_{31})/(100t_5\delta^3)]\partial/\partial s_{32} \\
&&+[(24t_2t_4s_{11}s_{21}^2s_{22}s_{41}-24t_2t_4s_{12}s_{21}^3s_{41}-80t_2t_5s_{11}s_{21}^2s_{22}s_{42}+80t_2t_5s_{11}s_{21}s_{22}^2s_{41} \\
&&+80t_2t_5s_{12}s_{21}^3s_{42}-80t_2t_5s_{12}s_{21}^2s_{22}s_{41}-36t_3t_4s_{11}s_{21}s_{22}^2s_{41}+36t_3t_4s_{12}s_{21}^2s_{22}s_{41} \\
&&+120t_3t_5s_{11}s_{21}s_{22}^2s_{42}-120t_3t_5s_{11}s_{22}^3s_{41}-120t_3t_5s_{12}s_{21}^2s_{22}s_{42}+120t_3t_5s_{12}s_{21}s_{22}^2s_{41}\\
&&+48t_4^2s_{11}s_{22}^3s_{41}-48t_4^2s_{12}s_{21}s_{22}^2s_{41}-25t_4t_5s_{11}s_{21}s_{22}-60t_4t_5s_{11}s_{22}^3s_{42}-75t_4t_5s_{12}s_{21}^2\\
&&+60t_4t_5s_{12}s_{21}s_{22}^2s_{42}+50t_5^2s_{11}s_{22}^2+150t_5^2s_{12}s_{21}s_{22})/(400t_5\delta^3)] \partial/\partial s_{41} \\
&&+[(8t_2t_4s_{11}s_{21}^2s_{22}s_{42}-8t_2t_4s_{12}s_{21}^3s_{42}-40t_2t_5s_{11}s_{21}^2+160t_2t_5s_{11}s_{21}s_{22}^2s_{42} \\
&&-100t_2t_5s_{11}s_{22}^3s_{41}-160t_2t_5s_{12}s_{21}^2s_{22}s_{42}+100t_2t_5s_{12}s_{21}s_{22}^2s_{41}-12t_3t_4s_{11}s_{21}s_{22}^2s_{42} \\
&&+12t_3t_4s_{12}s_{21}^2s_{22}s_{42}+10t_3t_5s_{11}s_{21}s_{22}-40t_3t_5s_{11}s_{22}^3s_{42}-150t_3t_5s_{12}s_{21}^2\\
&&+40t_3t_5s_{12}s_{21}s_{22}^2s_{42}+16t_4^2s_{11}s_{22}^3s_{42}-16t_4^2s_{12}s_{21}s_{22}^2s_{42}-5t_4t_5s_{11}s_{22}^2+225t_4t_5s_{12}s_{21}s_{22}\\
&&-300t_5^2s_{12}s_{22}^2+100t_5s_{11}s_{21}^2s_{22}s_{41}-100t_5s_{12}s_{21}^3s_{41})/(400t_5\delta^3)]\partial/\partial s_{42}
\end{eqnarray*}

\begin{eqnarray*}
&& \Ra_2= \\
&& [(4t_2^2t_4s_{11}^2+30t_2^2t_5s_{11}s_{12}-6t_2t_3t_4s_{11}s_{12}-20t_2t_3t_5s_{12}^2+8t_2t_4^2s_{12}^2+75t_3t_5s_{11}^2 \\
&&-100t_4t_5s_{11}s_{12}+125t_5^2s_{12}^2)/(50t_5\delta^2)]\partial/\partial t_2 \\
&&+[(-30t_2^2t_5s_{11}^2+6t_2t_3t_4s_{11}^2+65t_2t_3t_5s_{11}s_{12}-10t_2t_4t_5s_{12}^2-9t_3^2t_4s_{11}s_{12}-30t_3^2t_5s_{12}^2\\&&+12t_3t_4^2s_{12}^2
+100t_4t_5s_{11}^2-125t_5^2s_{11}s_{12})/(50t_5\delta^2)] \partial/\partial t_3\\
&&+[(-20t_2t_3t_5s_{11}^2+8t_2t_4^2s_{11}^2+10t_2t_4t_5s_{11}s_{12}+75t_2t_5^2s_{12}^2+30t_3^2t_5s_{11}s_{12}-12t_3t_4^2s_{11}s_{12}\\
&&-55t_3t_4t_5s_{12}^2+16t_4^3s_{12}^2+125t_5^2s_{11}^2)/(50t_5\delta^2)] \partial/\partial t_4\\
&&+[0]\partial/\partial t_5 \\
&&+[(6t_2t_4s_{11}^3-9t_3t_4s_{11}^2s_{12}+12t_4^2s_{11}s_{12}^2-15t_4t_5s_{12}^3)/(100t_5\delta^2)] \partial/\partial s_{11}\\
&&+[(2t_2t_4s_{11}^2s_{12}+15t_2t_5s_{11}s_{12}^2-3t_3t_4s_{11}s_{12}^2-10t_3t_5s_{12}^3+4t_4^2s_{12}^3\\
&&+25t_5s_{11}^3)/(100t_5\delta^2)] \partial/\partial s_{12}\\
&&+[(6t_2t_4s_{11}^2s_{21}-20t_2t_5s_{11}^2s_{22}+20t_2t_5s_{11}s_{12}s_{21}-9t_3t_4s_{11}s_{12}s_{21}+30t_3t_5s_{11}s_{12}s_{22}\\
&&-30t_3t_5s_{12}^2s_{21}+12t_4^2s_{12}^2s_{21}-15t_4t_5s_{12}^2s_{22}+100t_5s_{11}^2s_{22}^2s_{41}-200t_5s_{11}s_{12}s_{21}s_{22}s_{41}\\
&&+100t_5s_{12}^2s_{21}^2s_{41})/(100t_5\delta^2)] \partial/\partial s_{21}\\
&&+[(2t_2t_4s_{11}^2s_{22}+40t_2t_5s_{11}s_{12}s_{22}-25t_2t_5s_{12}^2s_{21}-3t_3t_4s_{11}s_{12}s_{22}-10t_3t_5s_{12}^2s_{22}\\
&&+4t_4^2s_{12}^2s_{22}+25t_5s_{11}^2s_{21}+100t_5s_{11}^2s_{22}^2s_{42}-200t_5s_{11}s_{12}s_{21}s_{22}s_{42}+100t_5s_{12}^2s_{21}^2s_{42})/\\
&&(100t_5\delta^2)] \partial/\partial s_{22}\\
&&+[(24t_2t_4s_{11}^3s_{22}s_{31}-24t_2t_4s_{11}^2s_{12}s_{21}s_{31}-80t_2t_5s_{11}^3s_{22}s_{32}+80t_2t_5s_{11}^2s_{12}s_{21}s_{32}\\
&&+80t_2t_5s_{11}^2s_{12}s_{22}s_{31}-80t_2t_5s_{11}s_{12}^2s_{21}s_{31}-36t_3t_4s_{11}^2s_{12}s_{22}s_{31}+36t_3t_4s_{11}s_{12}^2s_{21}s_{31}\\
&&+120t_3t_5s_{11}^2s_{12}s_{22}s_{32}-120t_3t_5s_{11}s_{12}^2s_{21}s_{32}-120t_3t_5s_{11}s_{12}^2s_{22}s_{31}+120t_3t_5s_{12}^3s_{21}s_{31}\\
&&+48t_4^2s_{11}s_{12}^2s_{22}s_{31}-48t_4^2s_{12}^3s_{21}s_{31}+75t_4t_5s_{11}^2s_{22}-60t_4t_5s_{11}s_{12}^2s_{22}s_{32}+25t_4t_5s_{11}s_{12}s_{21}\\
&&+60t_4t_5s_{12}^3s_{21}s_{32}-150t_5^2s_{11}s_{12}s_{22}-50t_5^2s_{12}^2s_{21})/(400t_5\delta^3)] \partial/\partial s_{31}\\
&&+[(8t_2t_4s_{11}^3s_{22}s_{32}-8t_2t_4s_{11}^2s_{12}s_{21}s_{32}+160t_2t_5s_{11}^2s_{12}s_{22}s_{32}+40t_2t_5s_{11}^2s_{21} \\
&&-160t_2t_5s_{11}s_{12}^2s_{21}s_{32}-100t_2t_5s_{11}s_{12}^2s_{22}s_{31}+100t_2t_5s_{12}^3s_{21}s_{31}-12t_3t_4s_{11}^2s_{12}s_{22}s_{32}\\
&&+12t_3t_4s_{11}s_{12}^2s_{21}s_{32}+150t_3t_5s_{11}^2s_{22}-40t_3t_5s_{11}s_{12}^2s_{22}s_{32}-10t_3t_5s_{11}s_{12}s_{21}+40t_3t_5s_{12}^3s_{21}s_{32}\\
&&+16t_4^2s_{11}s_{12}^2s_{22}s_{32}-16t_4^2s_{12}^3s_{21}s_{32}-225t_4t_5s_{11}s_{12}s_{22}+5t_4t_5s_{12}^2s_{21}+300t_5^2s_{12}^2s_{22}\\
&&+100t_5s_{11}^3s_{22}s_{31}-100t_5s_{11}^2s_{12}s_{21}s_{31})/(400t_5\delta^3]\partial/\partial s_{32} \\
&&+[(6t_2t_4s_{11}^3s_{22}s_{41}-6t_2t_4s_{11}^2s_{12}s_{21}s_{41}-20t_2t_5s_{11}^3s_{22}s_{42}+20t_2t_5s_{11}^2s_{12}s_{21}s_{42}\\
&&+20t_2t_5s_{11}^2s_{12}s_{22}s_{41}-20t_2t_5s_{11}s_{12}^2s_{21}s_{41}-9t_3t_4s_{11}^2s_{12}s_{22}s_{41}+9t_3t_4s_{11}s_{12}^2s_{21}s_{41}\\
&&+30t_3t_5s_{11}^2s_{12}s_{22}s_{42}-30t_3t_5s_{11}s_{12}^2s_{21}s_{42}-30t_3t_5s_{11}s_{12}^2s_{22}s_{41}+30t_3t_5s_{12}^3s_{21}s_{41}\\
&&+12t_4^2s_{11}s_{12}^2s_{22}s_{41}-12t_4^2s_{12}^3s_{21}s_{41}-25t_4t_5s_{11}^2s_{12}-15t_4t_5s_{11}s_{12}^2s_{22}s_{42}+15t_4t_5s_{12}^3s_{21}s_{42}\\
&&+50t_5^2s_{11}s_{12}^2)/(100t_5\delta^3)] \partial/\partial s_{41}\\
&&+[(2t_2t_4s_{11}^3s_{22}s_{42}-2t_2t_4s_{11}^2s_{12}s_{21}s_{42}-10t_2t_5s_{11}^3+40t_2t_5s_{11}^2s_{12}s_{22}s_{42}-40t_2t_5s_{11}s_{12}^2s_{21}s_{42}\\
&&-25t_2t_5s_{11}s_{12}^2s_{22}s_{41}+25t_2t_5s_{12}^3s_{21}s_{41}-3t_3t_4s_{11}^2s_{12}s_{22}s_{42}+3t_3t_4s_{11}s_{12}^2s_{21}s_{42}-35t_3t_5s_{11}^2s_{12}\\
&&-10t_3t_5s_{11}s_{12}^2s_{22}s_{42}+10t_3t_5s_{12}^3s_{21}s_{42}+4t_4^2s_{11}s_{12}^2s_{22}s_{42}-4t_4^2s_{12}^3s_{21}s_{42}+55t_4t_5s_{11}s_{12}^2\\
&&-75t_5^2s_{12}^3+25t_5s_{11}^3s_{22}s_{41}-25t_5s_{11}^2s_{12}s_{21}s_{41})/(100t_5\delta^3)] \partial/\partial s_{42}
\end{eqnarray*}

\begin{eqnarray*}
&& \Ra_3=\\
&& [(-4t_2^2t_4s_{11}s_{21}-15t_2^2t_5s_{11}s_{22}-15t_2^2t_5s_{12}s_{21}+3t_2t_3t_4s_{11}s_{22}+3t_2t_3t_4s_{12}s_{21} \\
&&+20t_2t_3t_5s_{12}s_{22}-8t_2t_4^2s_{12}s_{22}-75t_3t_5s_{11}s_{21}+50t_4t_5s_{11}s_{22}+50t_4t_5s_{12}s_{21}\\
&&-125t_5^2s_{12}s_{22})/(25t_5\delta^2)] \partial/\partial t_2\\
&&+[(60t_2^2t_5s_{11}s_{21}-12t_2t_3t_4s_{11}s_{21}-65t_2t_3t_5s_{11}s_{22}-65t_2t_3t_5s_{12}s_{21}+20t_2t_4t_5s_{12}s_{22}\\
&&+9t_3^2t_4s_{11}s_{22}+9t_3^2t_4s_{12}s_{21}+60t_3^2t_5s_{12}s_{22}-24t_3t_4^2s_{12}s_{22}-200t_4t_5s_{11}s_{21}+125t_5^2s_{11}s_{22}\\
&&+125t_5^2s_{12}s_{21})/(50t_5\delta^2)] \partial/\partial t_3 \\
&&+[(20t_2t_3t_5s_{11}s_{21}-8t_2t_4^2s_{11}s_{21}-5t_2t_4t_5s_{11}s_{22}-5t_2t_4t_5s_{12}s_{21}-75t_2t_5^2s_{12}s_{22} \\
&&-15t_3^2t_5s_{11}s_{22}-15t_3^2t_5s_{12}s_{21}+6t_3t_4^2s_{11}s_{22}+6t_3t_4^2s_{12}s_{21}+55t_3t_4t_5s_{12}s_{22}-16t_4^3s_{12}s_{22}\\
&&-125t_5^2s_{11}s_{21})/(25t_5\delta^2)] \partial/\partial t_4 \\
&&+[0]\partial/\partial t_5 \\
&&+[(-12t_2t_4s_{11}^2s_{21}-20t_2t_5s_{11}^2s_{22}+20t_2t_5s_{11}s_{12}s_{21}+9t_3t_4s_{11}^2s_{22}+9t_3t_4s_{11}s_{12}s_{21}\\
&&+30t_3t_5s_{11}s_{12}s_{22}-30t_3t_5s_{12}^2s_{21}-24t_4^2s_{11}s_{12}s_{22}+30t_4t_5s_{12}^2s_{22}+100t_5s_{11}^2s_{22}^2s_{41}\\
&&-200t_5s_{11}s_{12}s_{21}s_{22}s_{41}+100t_5s_{12}^2s_{21}^2s_{41})/(100t_5\delta^2)]\partial/\partial s_{11} \\
&&+[(-4t_2t_4s_{11}s_{12}s_{21}+10t_2t_5s_{11}s_{12}s_{22}-40t_2t_5s_{12}^2s_{21}+3t_3t_4s_{11}s_{12}s_{22}+3t_3t_4s_{12}^2s_{21}\\
&&+20t_3t_5s_{12}^2s_{22}-8t_4^2s_{12}^2s_{22}-50t_5s_{11}^2s_{21}+100t_5s_{11}^2s_{22}^2s_{42}-200t_5s_{11}s_{12}s_{21}s_{22}s_{42}\\
&&+100t_5s_{12}^2s_{21}^2s_{42})/
(100t_5\delta^2)]\partial/\partial s_{12} \\
&&+[(-12t_2t_4s_{11}s_{21}^2+20t_2t_5s_{11}s_{21}s_{22}-20t_2t_5s_{12}s_{21}^2+9t_3t_4s_{11}s_{21}s_{22}+9t_3t_4s_{12}s_{21}^2\\
&&-30t_3t_5s_{11}s_{22}^2+30t_3t_5s_{12}s_{21}s_{22}-24t_4^2s_{12}s_{21}s_{22}+30t_4t_5s_{12}s_{22}^2+100t_5s_{11}^2s_{22}^2s_{31}\\
&&-200t_5s_{11}s_{12}s_{21}s_{22}s_{31}+100t_5s_{12}^2s_{21}^2s_{31})/(100t_5\delta^2)]\partial/\partial s_{21} \\
&&+[(-4t_2t_4s_{11}s_{21}s_{22}-40t_2t_5s_{11}s_{22}^2+10t_2t_5s_{12}s_{21}s_{22}+3t_3t_4s_{11}s_{22}^2+3t_3t_4s_{12}s_{21}s_{22}\\
&&+20t_3t_5s_{12}s_{22}^2-8t_4^2s_{12}s_{22}^2+100t_5s_{11}^2s_{22}^2s_{32}-200t_5s_{11}s_{12}s_{21}s_{22}s_{32}-50t_5s_{11}s_{21}^2\\
&&+100t_5s_{12}^2s_{21}^2s_{32})/(100t_5\delta^2)] \partial/\partial s_{22} \\
&&+[(-48t_2t_4s_{11}^2s_{21}s_{22}s_{31}+48t_2t_4s_{11}s_{12}s_{21}^2s_{31}+160t_2t_5s_{11}^2s_{21}s_{22}s_{32}-80t_2t_5s_{11}^2s_{22}^2s_{31} \\
&&-160t_2t_5s_{11}s_{12}s_{21}^2s_{32}+80t_2t_5s_{12}^2s_{21}^2s_{31}+36t_3t_4s_{11}^2s_{22}^2s_{31}-36t_3t_4s_{12}^2s_{21}^2s_{31}\\
&&-120t_3t_5s_{11}^2s_{22}^2s_{32}+240t_3t_5s_{11}s_{12}s_{22}^2s_{31}+120t_3t_5s_{12}^2s_{21}^2s_{32}-240t_3t_5s_{12}^2s_{21}s_{22}s_{31}\\
&&-96t_4^2s_{11}s_{12}s_{22}^2s_{31}+96t_4^2s_{12}^2s_{21}s_{22}s_{31}+120t_4t_5s_{11}s_{12}s_{22}^2s_{32}-175t_4t_5s_{11}s_{21}s_{22}\\
&&-120t_4t_5s_{12}^2s_{21}s_{22}s_{32}-25t_4t_5s_{12}s_{21}^2+150t_5^2s_{11}s_{22}^2+250t_5^2s_{12}s_{21}s_{22})/(400t_5\delta^3)] \partial/\partial s_{31} \\
&&+[(-16t_2t_4s_{11}^2s_{21}s_{22}s_{32}+16t_2t_4s_{11}s_{12}s_{21}^2s_{32}-160t_2t_5s_{11}^2s_{22}^2s_{32}+200t_2t_5s_{11}s_{12}s_{22}^2s_{31}\\&&-80t_2t_5s_{11}s_{21}^2+160t_2t_5s_{12}^2s_{21}^2s_{32}-200t_2t_5s_{12}^2s_{21}s_{22}s_{31}+12t_3t_4s_{11}^2s_{22}^2s_{32}-12t_3t_4s_{12}^2s_{21}^2s_{32}\\
&&+80t_3t_5s_{11}s_{12}s_{22}^2s_{32}-290t_3t_5s_{11}s_{21}s_{22}-80t_3t_5s_{12}^2s_{21}s_{22}s_{32}+10t_3t_5s_{12}s_{21}^2\\
&&-32t_4^2s_{11}s_{12}s_{22}^2s_{32}+32t_4^2s_{12}^2s_{21}s_{22}s_{32}+225t_4t_5s_{11}s_{22}^2+215t_4t_5s_{12}s_{21}s_{22}-600t_5^2s_{12}s_{22}^2\\
&&-200t_5s_{11}^2s_{21}s_{22}s_{31}+200t_5s_{11}s_{12}s_{21}^2s_{31})/(400t_5\delta^3)]\partial/\partial s_{32} \\
&&+[(-48t_2t_4s_{11}^2s_{21}s_{22}s_{41}+48t_2t_4s_{11}s_{12}s_{21}^2s_{41}+160t_2t_5s_{11}^2s_{21}s_{22}s_{42}-80t_2t_5s_{11}^2s_{22}^2s_{41}\\
&&-160t_2t_5s_{11}s_{12}s_{21}^2s_{42}+80t_2t_5s_{12}^2s_{21}^2s_{41}+36t_3t_4s_{11}^2s_{22}^2s_{41}-36t_3t_4s_{12}^2s_{21}^2s_{41}\\
&&-120t_3t_5s_{11}^2s_{22}^2s_{42}+240t_3t_5s_{11}s_{12}s_{22}^2s_{41}+120t_3t_5s_{12}^2s_{21}^2s_{42}-240t_3t_5s_{12}^2s_{21}s_{22}s_{41}\\
&&-96t_4^2s_{11}s_{12}s_{22}^2s_{41}+96t_4^2s_{12}^2s_{21}s_{22}s_{41}+25t_4t_5s_{11}^2s_{22}+175t_4t_5s_{11}s_{12}s_{21}\\
&&+120t_4t_5s_{11}s_{12}s_{22}^2s_{42}-120t_4t_5s_{12}^2s_{21}s_{22}s_{42}-250t_5^2s_{11}s_{12}s_{22}-150t_5^2s_{12}^2s_{21})/\\
&&(400t_5\delta^3)]\partial/\partial s_{41} \\
\end{eqnarray*}

\begin{eqnarray*}
&&+[(-16t_2t_4s_{11}^2s_{21}s_{22}s_{42}+16t_2t_4s_{11}s_{12}s_{21}^2s_{42}+80t_2t_5s_{11}^2s_{21}-160t_2t_5s_{11}^2s_{22}^2s_{42}\\
&&+200t_2t_5s_{11}s_{12}s_{22}^2s_{41}+160t_2t_5s_{12}^2s_{21}^2s_{42}-200t_2t_5s_{12}^2s_{21}s_{22}s_{41}+12t_3t_4s_{11}^2s_{22}^2s_{42}\\
&&-12t_3t_4s_{12}^2s_{21}^2s_{42}-10t_3t_5s_{11}^2s_{22}+290t_3t_5s_{11}s_{12}s_{21}+80t_3t_5s_{11}s_{12}s_{22}^2s_{42}\\
&&-80t_3t_5s_{12}^2s_{21}s_{22}s_{42}-32t_4^2s_{11}s_{12}s_{22}^2s_{42}+32t_4^2s_{12}^2s_{21}s_{22}s_{42}-215t_4t_5s_{11}s_{12}s_{22}\\
&&-225t_4t_5s_{12}^2s_{21}+600t_5^2s_{12}^2s_{22}-200t_5s_{11}^2s_{21}s_{22}s_{41}+200t_5s_{11}s_{12}s_{21}^2s_{41})/(400t_5\delta^3)] \partial/\partial s_{42}\\
\end{eqnarray*}

\def\cprime{$'$} \def\cprime{$'$} \def\cprime{$'$} \def\cprime{$'$}




\end{document}